\documentclass[a4paper,12pt,reqno]{amsart}
\usepackage[T1]{fontenc}
\usepackage[utf8]{inputenc}
\usepackage{amssymb,bbm,mathrsfs}
\usepackage[margin=1in]{geometry}
\usepackage{enumitem}
\usepackage{multicol}
\usepackage{tikz-cd}
\usepackage[bookmarksdepth=2]{hyperref}

\newenvironment{leftbar}{}{}

\newtheorem{theorem}{Theorem}[section]
\newtheorem{lemma}[theorem]{Lemma}

\newtheorem{proposition}[theorem]{Proposition}

\theoremstyle{definition}
\newtheorem{definition}[theorem]{Definition}
\newtheorem{example}[theorem]{Example}
\newtheorem{remark}[theorem]{Remark}
\newtheorem{conjecture}[theorem]{Conjecture}

\DeclareMathOperator{\arsinh}{arsinh}
\DeclareMathOperator{\re}{Re}

\DeclareMathOperator{\sign}{sign}

\newcommand{\leb}{\mathscr{L}}
\newcommand{\cont}{\mathscr{C}}
\newcommand{\conto}{\cont_0}
\newcommand{\contb}{\cont_b}
\newcommand{\contbu}{\cont_{bu}}
\newcommand{\schw}{\mathscr{S}}
\newcommand{\dom}{\mathscr{D}}
\newcommand{\X}{\mathscr{X}}
\newcommand{\fourier}{\mathscr{F}}
\newcommand{\st}[1]{{\smash{\Hat{#1}}}}
\newcommand{\stt}[1]{{\smash{\check{#1}}}}
\newcommand{\tst}[1]{{$\smash{\Hat{\mbox{#1}}}$}}
\newcommand{\tstt}[1]{{$\smash{\check{\mbox{#1}}}$}}

\newcommand{\form}{\mathcal{E}}
\newcommand{\pr}{\mathbf{P}}
\newcommand{\ex}{\mathbf{E}}
\newcommand{\R}{\mathbf{R}}
\newcommand{\ind}{\mathbf{1}}
\newcommand{\ph}{\varphi}
\newcommand{\eps}{\varepsilon}
\newcommand{\tscalar}[1]{\langle #1 \rangle}
\newcommand{\set}[1]{\left\{ #1 \right\}}
\newcommand{\abs}[1]{\left| #1 \right|}
\newcommand{\expr}[1]{\left( #1 \right)}

\usepackage{ifthen}
\newcommand{\formula}[2][nolabel]%
{%
 \ifthenelse{\equal{#1}{nolabel}}%
 {\begin{align*} #2 \end{align*}}%
 {%
  \ifthenelse{\equal{#1}{}}%
  {\begin{align} #2 \end{align}}%
  {\begin{align} \label{#1} #2 \end{align}}%
 }%
}

\begin{document}

%
%

\title[Ten equivalent definitions of the fractional Laplace operator]{Ten equivalent definitions \\ of the fractional Laplace operator}
\author{Mateusz Kwaśnicki}
\thanks{Work supported from the statutory funds of Department of Pure and Applied Mathematics, Faculty of Fundamental Problems of Technology, Wrocław University of Technology}
\address{Mateusz Kwaśnicki \\ Department of Pure and Applied Mathematics \\ Wrocław University of Technology \\ ul. Wybrzeże Wyspiańskiego 27 \\ 50-370 Wrocław, Poland}
\email{mateusz.kwasnicki@pwr.edu.pl}
\date{\today}
\keywords{Fractional Laplacian; weak definition; Riesz potential; singular integral; extension technique; Bochner's subordination; Balakrishnan's formula; Dynkin's characteristic operator}
\subjclass[2010]{47G30; 35S05; 60J35}

\begin{abstract}
This article discusses several definitions of the fractional Laplace operator $(-\Delta)^{\alpha/2}$ ($\alpha \in (0, 2)$) in $\R^d$ ($d \ge 1$), also known as the Riesz fractional derivative operator, as an operator on Lebesgue spaces $\leb^p$ ($p \in [1, \infty)$), on the space $\conto$ of continuous functions vanishing at infinity and on the space $\contbu$ of bounded uniformly continuous functions. Among these definitions are ones involving singular integrals, semigroups of operators, Bochner's subordination and harmonic extensions. We collect and extend known results in order to prove that all these definitions agree: on each of the function spaces considered, the corresponding operators have common domain and they coincide on that common domain.
\end{abstract}

\maketitle

%
%

\section{Introduction}
\label{sec:intro}

We consider the fractional Laplace operator $L = -(-\Delta)^{\alpha/2}$ in $\R^d$, with $\alpha \in (0, 2)$ and $d \in \{1, 2, ...\}$. Numerous definitions of $L$ can be found in literature: as a Fourier multiplier with symbol $-|\xi|^\alpha$, as a fractional power in the sense of Bochner or Balakrishnan, as the inverse of the Riesz potential operator, as a singular integral operator, as an operator associated to an appropriate Dirichlet form, as an infinitesimal generator of an appropriate semigroup of contractions, or as the Dirichlet-to-Neumann operator for an appropriate harmonic extension problem. Equivalence of these definitions for sufficiently smooth functions is well-known and easy. The purpose of this article is to prove that, whenever meaningful, all these definitions are equivalent in the Lebesgue space $\leb^p$ for $p \in [1, \infty)$, in the space $\conto$ of continuous functions vanishing at infinity, and in the space $\contbu$ of bounded uniformly continuous functions. In order to achive this goal, we extend known results and simplify some proofs.

The literature on the above topic is rather scattered. Equivalence between semigroup definition, Bochner's formula and Balakrishnan's formula is a general result, see~\cite{bib:ms01}. Inversion of Riesz potentials is well-studied in the context of $\leb^p$ spaces (with $p \in [1, \tfrac{d}{\alpha})$) and certain classes of distributions, see~\cite{bib:l72,bib:r96,bib:s01,bib:s70}. Semigroup and singular integral definitions are also known to be equivalent, at least in $\leb^p$ for $p \in [1, \tfrac{d}{\alpha})$, see~\cite{bib:r96}. Finally, the semigroup definition on the space $\conto$ is known to be equivalent to Dynkin's pointwise definition, known as Dynkin's characteristic operator, see~\cite{bib:d65}.

In the present article we remove unnecessary restrictions and prove equivalence of the above definitions in full generality. Our proofs are elementary and mostly analytic, but they originate in potential theory (explicit expressions due to M.~Riesz) and Markov processes (Dynkin's characteristic operator). The main new ingredient consists of several relations between pointwise definitions of $L f(x)$, which are then re-used to prove norm convergence in $\leb^p$, $\conto$ and $\contbu$.

Note that we restrict $\alpha$ to $(0, 2)$, that is, we do not consider complex values of $\alpha$, nor we include the hypersingular case $\alpha > 2$. The following theorem summarises the results of the paper.

\begin{leftbar}
\begin{theorem}
\label{th:main}
Let $\X$ be any of the spaces $\leb^p$, $p \in [1, \infty)$, $\conto$ or $\contbu$, and let $f \in \X$. The following definitions of $L f \in \X$ are equivalent:
\begin{enumerate}[label={\rm (\alph*)}]
\item Fourier definition:
\formula{
 \fourier (L f)(\xi) = -|\xi|^\alpha \fourier f(\xi)
}
(if $\X = \leb^p$, $p \in [1, 2]$);
\item distributional definition:
\formula{
 \int_{\R^d} L f(y) \ph(y) dy = \int_{\R^d} f(x) L \ph(x) dx
}
for all Schwartz functions $\ph$, with $L \ph$ defined, for example, as in~(a);
\item Bochner's definition:
\formula{
 L f & = \frac{1}{|\Gamma(-\tfrac{\alpha}{2})|} \int_0^\infty (e^{t \Delta} f - f) t^{-1 - \alpha/2} dt ,
}
with the Bochner's integral of an $\X$-valued function;
\item Balakrishnan's definition:
\formula{
 L f & = \frac{\sin \tfrac{\alpha \pi}{2}}{\pi} \int_0^\infty \Delta (s I - \Delta)^{-1} f \, s^{\alpha/2 - 1} ds ,
}
with the Bochner's integral of an $\X$-valued function;
\item singular integral definition:
\formula{
 L f & = \lim_{r \to 0^+} \frac{2^\alpha \Gamma(\tfrac{d + \alpha}{2})}{\pi^{d/2} |\Gamma(-\tfrac{\alpha}{2})|} \int_{\R^d \setminus B(x, r)} \frac{f(\cdot + z) - f(\cdot)}{|z|^{d + \alpha}} \, dz ,
}
with the limit in $\X$;
\item Dynkin's definition:
\formula{
 L f & = \lim_{r \to 0^+} \frac{2^\alpha \Gamma(\tfrac{d + \alpha}{2})}{\pi^{d/2} |\Gamma(-\tfrac{\alpha}{2})|} \int_{\R^d \setminus \overline{B}(x, r)} \frac{f(\cdot + z) - f(\cdot)}{|z|^d (|z|^2 - r^2)^{\alpha/2}} \, dz ,
}
with the limit in $\X$;
\item quadratic form definition: $\tscalar{L f, \ph} = \form(f, \ph)$ for all $\ph$ in the Sobolev space $H^{\alpha/2}$, where
\formula{
 \form(f, g) & = \frac{2^\alpha \Gamma(\tfrac{d + \alpha}{2})}{2 \pi^{d/2} |\Gamma(-\tfrac{\alpha}{2})|} \int_{\R^d} \int_{\R^d} \frac{(f(y) - f(x)) (\overline{g(y)} - \overline{g(x)})}{|x - y|^{d + \alpha}} \, dx dy
}
(if $\X = \leb^2$);
\item semigroup definition:
\formula{
 L f & = \lim_{t \to 0^+} \frac{P_t f - f}{t} ,
}
where $P_t f = f * p_t$ and $\fourier p_t(\xi) = e^{-t |\xi|^\alpha}$;
\item definition as the inverse of the Riesz potential:
\formula{
 \frac{\Gamma(\tfrac{d - \alpha}{2})}{2^\alpha \pi^{d/2} \Gamma(\tfrac{\alpha}{2})} \int_{\R^d} \frac{L f(\cdot + z)}{|z|^{d - \alpha}} \, dz = -f(\cdot)
}
(if $\alpha < d$ and $\X = \leb^p$, $p \in [1, \tfrac{d}{\alpha})$);
\item definition through harmonic extensions:
\formula{
 \begin{cases}
 \Delta_x u(x, y) + \alpha^2 c_\alpha^{2 / \alpha} y^{2 - 2/\alpha} \partial_y^2 u(x, y) = 0 & \text{for $y > 0$,} \\
 u(x, 0) = f(x) , \\
 \partial_y u(x, 0) = L f(x) ,
 \end{cases}
}
where $c_\alpha = 2^{-\alpha} |\Gamma(-\tfrac{\alpha}{2})| / \Gamma(\tfrac{\alpha}{2})$ and where $u(\cdot, y)$ is a function of class $\X$ which depends continuously on $y \in [0, \infty)$ and $\|u(\cdot, y)\|_{\X}$ is bounded in $y \in [0, \infty)$. 
\end{enumerate}
In addition, in~(c), (e), (f), (h) and (j), convergence in the uniform norm can be relaxed to pointwise convergence to a function in $\X$ when $\X = \conto$ or $\X = \contbu$. Finally, for $\X = \leb^p$ with $p \in [1, \infty)$, norm convergence in~(e), (f), (h) or (j) implies pointwise convergence for almost all $x$.
\end{theorem}
\end{leftbar}

We emphasize that in each definition, both $f$ and $L f$ are assumed to be in $\X$. For the detailed statements of the above definitions, we refer the reader to Section~\ref{sec:def}. Theorem~\ref{th:main} is a combination of Theorems~\ref{th:main1} and~\ref{th:main2}, and Lemma~\ref{lem:aed}.

\begin{remark}
The main novelty of the proof lies in the following observation: for a fixed $x$, the convergence in the Dynkin's definition~(h) implies convergence in the singular integral definition~(e), which in turn asserts convergence in the semigroup and harmonic extension definitions~(h) and~(j). By using this property instead of more advanced techniques found in literature, such as inversion of Riesz potentials, we are able to cover all $\leb^p$ spaces, with no restrictions on $p \in [1, \infty)$, as well as $\conto$ and $\contbu$. Probabilistic methods are perfectly suited to prove convergence in the Dynkin's definition~(h), see Section~\ref{sec:process} for examples and further discussion.
\end{remark}

\begin{remark}
The spaces $\X = \leb^\infty$ and $\X = \contb$ (the space of bounded continuous functions), not included in the theorem, can often be reduced to $\contbu$ due to the fact that the domains of $L$ defined with the singular integral, Dynkin's, semigroup or harmonic extension definitions on $\contbu$, $\contb$ and $\leb^\infty$ are all equal. Indeed, for example, consider the Dynkin's definition of $L$ on $\leb^\infty$. The convolution of any $f \in \leb^\infty$ with kernel $\tilde{\nu}_r(z) = |z|^{-d} (r^2 - |z|^2)^{-\alpha/2} \ind_{\R^d \setminus \overline{B}_r}(z)$ is uniformly continuous. Therefore, if the limit $L f$ in the Dynkin's definition exists in $\leb^\infty$ norm, then
\formula{
 L f = \lim_{r \to 0^+} (f * \tilde{\nu}_r - \|\tilde{\nu}_r\|_1 f)
}
(with the limit in $\leb^\infty$), so that in particular
\formula{
 f = \lim_{r \to 0^+} \frac{f * \tilde{\nu}_r}{\|\tilde{\nu}_r\|_1}
}
(again with the limit in $\leb^\infty$). Since the limit in $\leb^\infty$ of uniformly continuous functions is uniformly continuous, we conclude that $f \in \contbu$, and thus also $L f \in \contbu$.

An essential extension of $L$ defined with, for example, semigroup definition~(h) is possible when uniform convergence in $\contbu$ is replaced with uniform convergence on all compact subsets of $\R^d$ in $\contb$. For a detailed discussion of this concept in a much more general context of $\contb$-Feller semigroups, see~\cite[Section~4.8]{bib:j01}.
\end{remark}

We remark that in this article we restrict our attention to convergence problems for full-space definitions of the fractional Laplace operator $L$. Various results are also known for the restriction of $L$ to a domain: detailed properties of $L$ in a domain with Dirichlet condition can be found in~\cite{bib:bbkrsv09,bib:bkk08,bib:bkk15,bib:bz06,bib:k11,bib:r38a}; see also~\cite{bib:bbc03} for the study of $L$ in a domain with certain reflection. Explicit expressions for the fractional Laplace operator can be found in~\cite{bib:bgr61,bib:bkk08,bib:bz06,bib:d12,bib:dkk15,bib:g61,bib:r38a}. We also refer to a survey article~\cite{bib:ro15} and the references therein for an analytical perspective. Some applications of the fractional Laplace operator are reviewed in~\cite{bib:bv15}.  

The article is organised as follows. In Section~\ref{sec:def} we collect various definitions of the fractional Laplace operator $L$. Pointwise definitions~(c), (e), (f), (h) and~(j) for a fixed $x$ are studied Section~\ref{sec:pointwise}. In Section~\ref{sec:riesz}, M.~Riesz's explicit expressions for the harmonic measure (or the Poisson kernel) and the Green function of a ball are used to identify norm convergence in~(e), (f) and (h). Further results for norm convergence are given in Section~\ref{sec:norm}, where the main part of Theorem~\ref{th:main} is proved. Section~\ref{sec:other} collects further results: equivalence of pointwise and uniform convergence in $\conto$ and $\contbu$, almost everywhere convergence in $\leb^p$, and a sample regularity result for $L$. Finally, in Section~\ref{sec:process} we discuss the probabilistic definition of $L$ involving an isotropic $\alpha$-stable Lévy process.

The following notation is used throughout the article. 
By $\schw$ we denote the class of Schwartz functions, and $\schw'$ is the space of Schwartz distributions. Fourier transform of an integrable function $f$ is a $\conto$ function $\fourier f$ defined by
\formula{
 \fourier f(\xi) & = \int_{\R^d} e^{-i \xi \cdot x} f(x) dx .
}
Fourier transform extends continuously to an operator from $\leb^p$ to $\leb^q$ for $p \in [1, 2]$ and $\tfrac{1}{p} + \tfrac{1}{q} = 1$. It also maps $\schw$ to $\schw$, and by duality it extends to a mapping from $\schw'$ to $\schw'$. The Gauss--Weierstrass kernel (or the heat kernel) $k_t(x)$ is defined by
\formula{
 k_t(x) & = (4 \pi t)^{-d/2} \exp(-|x|^2 / (4 t)) , & \fourier k_t(\xi) & = e^{-t |\xi|^2} .
}
We denote $B_r = B(0, r)$ and $B = B_1$. Generic positive constants are denoted by $C$ (or $C(d)$, $C(d, \alpha)$ etc.). Occasionally we use the Gauss hypergeometric function $_2F_1$.

%
%

\section{Definitions of fractional Laplace operator}
\label{sec:def}

In this section we review different definitions of $L$. For each of them we discuss basic properties and provide an informal motivation.

%
%

\subsection{Fourier transform}
\label{sec:F}

In the setting of Hilbert spaces, the fractional power of a self-adjoint operator is typically defined by means of spectral theory. Observe that the Laplace operator $\Delta$ takes diagonal form in the Fourier variable: it is a Fourier multiplier with symbol $-|\xi|^2$, namely $\fourier (\Delta f)(\xi) = -|\xi|^2 \fourier f(\xi)$ for $f \in \schw$. By the spectral theorem, the operator $L = -(-\Delta)^{\alpha/2}$ also takes diagonal form in the Fourier variable: it is a Fourier multiplier with symbol $-|\xi|^\alpha$.

\begin{leftbar}
\begin{definition}[Fourier transform definition of $L$]
The fractional Laplace operator is given by
\formula{
\tag{F}\label{eq:F}
 \fourier (L_F f)(\xi) & = -|\xi|^\alpha \fourier f(\xi) .
}
More formally, let $\X = \leb^p$, where $p \in [1, 2]$. We say that $f \in \dom(L_F, \X)$ whenever $f \in \X$ and there is $L_F f \in \X$ such that~\eqref{eq:F} holds.
\end{definition}
\end{leftbar}

%
%

\subsection{Weak formulation}
\label{sec:W}

The Fourier transform of a convolution of two functions is the product of their Fourier transforms, and $-|\xi|^\alpha$ is a Schwartz distribution, so it is the Fourier transform of some $\tilde{L} \in \schw'$. Therefore, the fractional Laplace operator $L$ is the convolution operator with kernel $\tilde{L}$. Unfortunately, the convolution of two Schwartz distributions is not always well-defined. Nevertheless, the following definition is a very general one. 

\begin{leftbar}
\begin{definition}[distributional definition of $L$]
Let $\tilde{L}$ be the distribution in $\schw'$ with Fourier transform $-|\xi|^\alpha$. The weak (or distributional) fractional Laplace operator is given by
\formula{
\tag{W}\label{eq:W}
 L_W f & = \tilde{L} * f .
}
We write $f \in \dom(L_W, \schw')$ whenever $f \in \schw'$ and the convolution of $\tilde{L}$ and $f$ is well-defined in $\schw'$, that is, for all $\ph, \psi \in \schw$ the functions $\tilde{L} * \ph$ and $f * \psi$ are convolvable in the usual sense, and
\formula[eq:W:def]{
 L_W f * (\ph * \psi) & = (\tilde{L} * \ph) * (f * \psi) .
}
When $f$ and $L_W f$ both belong to $\X$, we write $f \in \dom(L_W, \X)$.
\end{definition}
\end{leftbar}

Note that $f \in \dom(L_W, \X)$ if and only if $f \in \X$ and there is $L_W f \in \X$ such that
\formula{
 L_W f * \psi = (\tilde{L} * \ph) * f .
}
Indeed, the above equality follows from~\eqref{eq:W:def} by taking $\ph = \ph_n$ to be an approximate identity and passing to the limit. Conversely, a function $L_W f \in \X$ with the above property clearly satisfies~\eqref{eq:W:def}.

For a more detailed discussion of convolvability of Schwartz distributions in this context, see~\cite[Sections~2.1 and~2.2]{bib:k11} and the references therein. Distributional definition of $L$ is also studied in~\cite{bib:bb99,bib:l72,bib:s01}, see also~\cite{bib:b98,bib:s99}.

%
%

\subsection{Bochner's subordination and Balakrishnan's formula}
\label{sec:B}

If an operator generates a strongly continuous semigroup on a Banach space (and $L$ does, see Section~\ref{sec:S}), its fractional power can be defined using Bochner's subordination. Observe that $\lambda^{\alpha/2}$ is a Bernstein function with representation
\formula{
 \lambda^{\alpha/2} & = \frac{1}{|\Gamma(-\tfrac{\alpha}{2})|} \int_0^\infty (1 - e^{-t \lambda}) t^{-1 - \alpha/2} dt
}
(this identity follows easily by integrating by parts the integral for $\Gamma(1 - \tfrac{\alpha}{2})$). Therefore, at least in the sense of spectral theory on $\leb^2$,
\formula{
 L = -(-\Delta)^{\alpha/2} & = \frac{1}{|\Gamma(-\tfrac{\alpha}{2})|} \int_0^\infty (e^{t \Delta} - 1) t^{-1 - \alpha/2} dt .
}
Here $e^{t \Delta}$ is the convolution operator with the Gauss--Weierstrass kernel $k_t(z)$.

\begin{leftbar}
\begin{definition}[Bochner definition of $L$]
The definition of the fractional Laplace operator through Bochner's subordination is given by
\formula{
\tag{B}\label{eq:B}
\begin{aligned}
 L_B f(x) & = \frac{1}{|\Gamma(-\tfrac{\alpha}{2})|} \int_0^\infty (f * k_t(x) - f(x)) t^{-1 - \alpha/2} dt \\
 & = \frac{1}{|\Gamma(-\tfrac{\alpha}{2})|} \int_0^\infty \expr{\int_{\R^d} (f(x + z) - f(x)) k_t(z) dz} t^{-1 - \alpha/2} dt .
\end{aligned}
}
More formally, we say that $f \in \dom(L_B, x)$ if the integrals converge for a given $x \in \R^d$. If $\X$ is a Banach space, $f \in \X$ and $\|f * k_t - f\|_{\X} \, t^{-1 - \alpha/2}$ is integrable in $t \in (0, \infty)$, then the first expression for $L_B f(x)$ in~\eqref{eq:B} can be understood as the Bochner's integral of a function with values in $\X$, and in this case we write $f \in \dom(L_B, \X)$.
\end{definition}
\end{leftbar}

Note that in general the order of integration in~\eqref{eq:B} cannot be changed. In particular, \eqref{eq:B} is likely the only way to interpret $L f(x)$ for some functions $f$ for which $(1 + |x|)^{-d - \alpha} f(x)$ is not integrable. The most natural examples here are harmonic polynomials: if $f$ is a harmonic polynomial, then $f * k_t(x) = f(x)$, and hence, according to~\eqref{eq:B}, $L_B f(x) = 0$ for all $x$.

A closely related approach uses the representation of $\lambda^{\alpha/2}$ as a complete Bernstein function, or an operator monotone function,
\formula{
 \lambda^{\alpha/2} & = \frac{\sin \tfrac{\alpha \pi}{2}}{\pi} \int_0^\infty \frac{\lambda}{s + \lambda} \, s^{\alpha/2 - 1} ds
}
(an identity which is typically proved using complex variable methods, or by a substitution $s = \lambda (1 - t) / t$, which reduces it to a beta integral). This way of defining the fractional power of a dissipative operator was introduced by Balakrishnan. Recall that $(s I - \Delta)^{-1}$, the resolvent of $\Delta$, is a Fourier multiplier with symbol $(|\xi|^2 + s)^{-1}$ and a convolution operator with kernel function $(2 \pi)^{-d/2} (\sqrt{s} \, x)^{1-d/2} K_{d/2-1}(\sqrt{s} \, x)$, where $K_{d/2-1}$ is the modified Bessel function of the second kind.

\begin{leftbar}
\begin{definition}[Balakrishnan definition of $L$]
The definition of the fractional Laplace operator as the Balakrishnan's fractional power is given by
\formula{
\tag{\tst{B}}\label{eq:BB}
\begin{aligned}
 L_\st{B} f(x) & = \frac{\sin \tfrac{\alpha \pi}{2}}{\pi} \int_0^\infty \Delta (s I - \Delta)^{-1} f(x) s^{\alpha/2 - 1} ds .
\end{aligned}
}
More formally, if $\X$ is a Banach space on which $\Delta$ has a strongly continuous family of resolvent operators $(s I - \Delta)^{-1}$ (with $s > 0$), $f \in \X$ and $\|\Delta (s I - \Delta)^{-1} f\|_\X s^{\alpha/2 - 1}$ is integrable in $s \in (0, \infty)$, then the integral in~\eqref{eq:BB} can be understood as the Bochner's integral of a function with values in $\X$, and in this case we write $f \in \dom(L_\st{B}, \X)$.
\end{definition}
\end{leftbar}

Among the Banach spaces considered in the present article, the family of resolvent operators is \emph{not} strongly continuous only when $\X = \leb^\infty$ or $\X = \contb$.

For more information on Bochner's subordination in the present context, see~\cite[Chapter~6]{bib:s99}, \cite[Chapter~13]{bib:ssv12} and the references therein. A complete treatment of the theory of fractional powers of operators, which includes the above two concepts in a much more general context, is given in~\cite{bib:ms01}.

%
%

\subsection{Singular integrals}
\label{sec:P}

If the order of integration in~\eqref{eq:B} could be reversed, we would have $L_B f(x) = \int_{\R^d} (f(x + z) - f(x)) \nu(z) dz$, where, by a substitution $t = |z|^2 / (4 s)$,
\formula[eq:nucalc]{
\begin{aligned}
 \nu(z) & = \frac{1}{|\Gamma(-\tfrac{\alpha}{2})|} \int_0^\infty k_t(z) t^{-1 - \alpha/2} dt \\
 & = \frac{1}{2^d \pi^{d/2} |\Gamma(-\tfrac{\alpha}{2})|} \int_0^\infty t^{-1 - (d + \alpha)/2} e^{-|z|^2 / (4 t)} dt \\
 & = \frac{2^\alpha}{\pi^{d/2} |\Gamma(-\tfrac{\alpha}{2})| \, |z|^{d + \alpha}} \int_0^\infty s^{-1 + (d + \alpha)/2} e^{-s} ds \\
 & = \frac{2^\alpha \Gamma(\tfrac{d + \alpha}{2})}{\pi^{d/2} |\Gamma(-\tfrac{\alpha}{2})| \, |z|^{d + \alpha}} \, .
\end{aligned}
}
This motivates the classical pointwise definition of the fractional Laplace operator.

\begin{leftbar}
\begin{definition}[singular integral definition of $L$]
The fractional Laplace operator is given by the Cauchy principal value integral
\formula{
\tag{I}\label{eq:I}
\begin{aligned}
 L_I f(x) & = \lim_{r \to 0^+} c_{d,\alpha} \int_{\R^d \setminus B_r} (f(x + z) - f(x)) \, \frac{1}{|z|^{d + \alpha}} \, dz \\
 & = \lim_{r \to 0^+} \int_{\R^d} (f(x + z) - f(x)) \nu_r(z) dz ,
\end{aligned}
}
where
\formula[eq:nu]{
 \nu_r(z) & = \frac{c_{d,\alpha}}{|z|^{d + \alpha}} \, \ind_{\R^d \setminus B_r}(z) ,
}
and
\formula[eq:cda]{
 c_{d,\alpha} & = \frac{2^\alpha \Gamma(\tfrac{d + \alpha}{2})}{\pi^{d/2} |\Gamma(-\tfrac{\alpha}{2})|} .
}
For later use, we denote $\nu(z) = \nu_0(z) = c_{d,\alpha} |z|^{-d - \alpha}$, and allow for negative $\alpha$ in~\eqref{eq:cda}. We write $f \in \dom(L_I, x)$ if the limit in~\eqref{eq:I} exists for a given $x \in \R^d$. We write $f \in \dom(L_I, \X)$ if $f \in \X$ and the limit in~\eqref{eq:I} exists in $\X$.
\end{definition}
\end{leftbar}

The following variant of~\eqref{eq:I} is commonly used in probability theory.

\begin{leftbar}
\begin{definition}[variant of the singular integral definition of $L$]
The fractional Laplace operator is given by
\formula{
\tag{\tst{I}}\label{eq:II}
\begin{aligned}
 L_\st{I} f(x) & = \int_{\R^d} (f(x + z) - f(x) - \nabla f(x) \cdot z \ind_B(z)) \nu(z) dz .
\end{aligned}
}
We write $f \in \dom(L_\st{I}, x)$ if $\nabla f(x)$ exists and the above integral converges absolutely.
\end{definition}
\end{leftbar}

A different regularisation of the singular integral~\eqref{eq:I} is sometimes found in analysis.

\begin{leftbar}
\begin{definition}[another variant of the singular integral definition of $L$]
The fractional Laplace operator is given by
\formula{
\tag{\tstt{I}}\label{eq:III}
\begin{aligned}
 L_\stt{I} f(x) & = \frac{1}{2} \int_{\R^d} (f(x + z) + f(x - z) - 2 f(x)) \nu(z) dz .
\end{aligned}
}
We write $f \in \dom(L_\stt{I}, x)$ if the above integral converges absolutely.
\end{definition}
\end{leftbar}

As we will see later, the following definition of $L$ as the Dynkin characteristic operator of the the isotropic $\alpha$-stable L\'evy process, although more complicated than~\eqref{eq:I}, has certain advantage. The notion of the Dynkin characteristic operator is discussed in more detail in Section~\ref{sec:process} below.

\begin{leftbar}
\begin{definition}[Dynkin definition of $L$]
The definition of the fractional Laplace operator as the Dynkin characteristic operator is given by
\formula{
\tag{D}\label{eq:D}
\begin{aligned}
 L_D f(x) & = \lim_{r \to 0^+} c_{d,\alpha} \int_{\R^d \setminus \overline{B}_r} (f(x + z) - f(x)) \, \frac{1}{|z|^d (|z|^2 - r^2)^{\alpha/2}} \, dz \\
 & = \lim_{r \to 0^+} \int_{\R^d} (f(x + z) - f(x)) \tilde{\nu}_r(z) dz ,
\end{aligned}
}
where
\formula[eq:tnu]{
 \tilde{\nu}_r(z) & = \frac{c_{d,\alpha}}{|z|^d (|z|^2 - r^2)^{\alpha/2}} \ind_{\R^d \setminus \overline{B}_r}(z) .
}
and $c_{d,\alpha}$ is given by~\eqref{eq:cda}. We write $f \in \dom(L_D, x)$ if the limit in~\eqref{eq:D} exists for a given $x \in \R^d$, and $f \in \dom(L_D, \X)$ if $f \in \X$ and the limit in~\eqref{eq:D} exists in $\X$.
\end{definition}
\end{leftbar}

%
%

\subsection{Quadratic forms}
\label{sec:Q}

A self-adjoint operator on $\leb^2$ is completely described by its quadratic form. By Fubini, for all $r > 0$ and all integrable $f, g$,
\formula{
 & \int_{\R^d} \expr{\int_{\R^d} (f(x + z) - f(x)) \nu_r(z) dz} \overline{g(x)} dx \\
 & \qquad = \int_{\R^d} \int_{\R^d} (f(y) - f(x)) \overline{g(x)} \nu_r(y - x) dy dx \\
 & \qquad = \frac{1}{2} \int_{\R^d} \int_{\R^d} (f(y) - f(x)) \overline{g(x)} \nu_r(y - x) dy dx \\
 & \qquad \qquad + \frac{1}{2} \int_{\R^d} \int_{\R^d} (f(x) - f(y)) \overline{g(y)} \nu_r(y - x) dy dx \\
 & \qquad = -\frac{1}{2} \int_{\R^d} \int_{\R^d} (f(y) - f(x)) (\overline{g(x)} - \overline{g(y)}) \nu_r(y - x) dy dx .
}
A formal limit as $r \to 0^+$ leads to the following definition.

\begin{leftbar}
\begin{definition}[quadratic form definition of $L$]
Let
\formula[eq:form]{
\begin{aligned}
 \form(f, g) & = \frac{c_{d,\alpha}}{2} \int_{\R^d} \int_{\R^d} \frac{(f(y) - f(x)) (\overline{g(y)} - \overline{g(x)})}{|x - y|^{d + \alpha}} \, dx dy \\
 & = \frac{1}{2} \int_{\R^d} \int_{\R^d} (f(y) - f(x)) (\overline{g(y)} - \overline{g(x)}) \nu(x - y) dx dy ,
\end{aligned}
}
where $\nu(z) = c_{d,\alpha} |z|^{-d - \alpha}$ and $c_{d,\alpha}$ is given by~\eqref{eq:cda}. We write $f \in \dom(\form)$ if $f \in \leb^2$ and $\form(f, f)$ is finite. We write $f \in \dom(L_Q, \leb^2)$ if $f \in \leb^2$ and there is $L_Q f \in \leb^2$ such that for all $g \in \dom(\form)$,
\formula{
\tag{Q}\label{eq:Q}
 \int_{\R^d} L_Q f(x) \overline{g(x)} dx & = -\form(f, g) .
}
\end{definition}
\end{leftbar}

We note that $\dom(\form)$ is the Sobolev space $H^{\alpha/2}$, which consists functions $f$ in $\leb^2$ such that $|\xi|^\alpha |\fourier f(\xi)|^2$ is integrable. Furthermore, $\form(f, g) = (2 \pi)^{-d} \int_{\R^d} \fourier |\xi|^\alpha f(\xi) \overline{\fourier g(\xi)} d\xi$. These properties follow easily from the expression for the Fourier transform of $p_t$ and the relation between $\form$ and $p_t$, see the proof of Lemma~\ref{lem:l2sq} below.

The quadratic form $\form$ is positive definite and it is an important example of a (non-local) Dirichlet form. For a detailed treatment of the theory of Dirichlet forms, see~\cite{bib:fot11}.

%
%

\subsection{Semigroup approach}
\label{sec:S}

By the spectral theorem, the fractional Laplace operator $L$ generates a strongly continuous semigroup of operators $P_t = e^{t L}$ on $\leb^2$, and $P_t$ is the Fourier multiplier with symbol $e^{-t |\xi|^\alpha}$. Hence, $P_t$ is the convolution operator with a symmetric kernel function $p_t(z)$, given by $\fourier p_t(\xi) = e^{-t |\xi|^\alpha}$.

We note some well-known properties of the kernel $p_t(z)$. For $\alpha = 1$, $p_t(z)$ is the Poisson kernel of the half-space in $\R^{d+1}$,
\formula[eq:cauchy]{
 p_t(z) & = \frac{\Gamma(\tfrac{d+1}{2})}{\pi^{(d + 1)/2}} \, \frac{t}{(t^2 + |z|^2)^{(d + 1)/2}} \, .
}
For arbitrary $\alpha \in (0, 2)$, $\fourier p_t$ is rapidly decreasing, and therefore $p_t$ is infinitely smooth. We also have $\fourier p_t(\xi) = \fourier p_1(t^{1/\alpha} \xi)$, and so
\formula[eq:ptscaling]{
 p_t(z) = t^{-d/\alpha} p_1(t^{-1/\alpha} z) .
}
Furthermore, by the Bochner's subordination formula, $p_t$ is the integral average of the Gauss--Weierstrass kernel $k_s$ with respect to $s \in (0, \infty)$. More precisely, let $\eta$ be a function on $(0, \infty)$ with Fourier--Laplace transform $\int_0^\infty e^{-\xi s} \eta(s) ds = \exp(-\xi^{\alpha/2})$ when $\re \xi > 0$. Then $\eta$ is smooth, positive, and $\eta(s)$ converges to $0$ as $s \to 0^+$ or $s \to \infty$, see~\cite[Remark~14.18]{bib:s99}. As one can easily verify using Fourier transform and Fubini,
\formula[eq:ptbochner]{
 p_t(z) & = t^{-2/\alpha} \int_0^\infty k_s(z) \eta(t^{-2/\alpha} s) ds .
}
Finally,
\formula[eq:heatlim]{
 \lim_{|z| \to \infty} |z|^{d + \alpha} p_1(z) & = c_{d,\alpha} , & \lim_{t \to 0^+} \frac{p_t(z)}{t} = c_{d,\alpha} |z|^{-d - \alpha}
}
with $c_{d,\alpha}$ given in~\eqref{eq:cda}. In particular
\formula[eq:ptest]{
 C_1(d,\alpha) \min(1, |z|^{-d - \alpha}) & \le p_1(z) \le C_2(d,\alpha) \min(1, |z|^{-d - \alpha}) .
}
Property~\eqref{eq:heatlim} can be easily derived using the Tauberian theory for the Laplace transform and~\eqref{eq:ptbochner}. Alternatively, one can use the Abelian--Tauberian theory for the Fourier transform (or, more precisely, for the corresponding Hankel transform), see~\cite{bib:l83,bib:ss75} and the references therein. Yet another way to show~\eqref{eq:heatlim} involves vague convergence of $t^{-1} p_t(z) dz$ to $\nu(z) dz = c_{d,\alpha} |z|^{-d - \alpha} dz$ as $t \to 0^+$, which is a general result in the theory of convolution semigroups, see~\cite{bib:s99}. Pointwise convergence is then a consequence of appropriate regularity of $p_t(z)$. We omit the details here, and refer to Lemma~\ref{lem:pis} below for a formal proof of a more detailed property of $p_t(z)$.

\begin{leftbar}
\begin{definition}[semigroup definition of $L$]
The fractional Laplace operator is given by
\formula{
\tag{S}\label{eq:S}
\begin{aligned}
 L_S f(x) & = \lim_{t \to 0^+} \frac{P_t f(x) - f(x)}{t} \\
 & = \lim_{t \to 0^+} \int_{\R^d} (f(x + z) - f(x)) \, \frac{p_t(z)}{t} \, dz .
\end{aligned}
}
Here $\fourier p_t(\xi) = e^{-t |\xi|^\alpha}$ and $P_t f(x) = f * p_t(x)$ (note that $p_t(z) = p_t(-z)$). More precisely, we write $f \in \dom(L_S, x)$ if the limit exists for a given $x \in \R^d$. If $f \in \X$ and the limit~\eqref{eq:S} exists in $\X$, then we write $f \in \dom(L_S, \X)$.
\end{definition}
\end{leftbar}

The above approach is distinguished due to the general theory of strongly continuous semigroups of operators on Banach spaces. The operators $P_t$ form a semigroup of contractions on every $\leb^p$, $p \in [1, \infty]$, and on $\conto$, $\contbu$ and $\contb$. This semigroup is strongly continuous, except on $\leb^\infty$ and $\contb$.

Suppose that $\X$ is any of the spaces $\leb^p$, $p \in [1, \infty)$, $\conto$ or $\contbu$. By the general theory, for $\lambda > 0$ the operator $\lambda I - L_S$ (with $L_S$ defined by~\eqref{eq:S} with convergence in $\X$) is a bijective map from $\dom(L_S, \X)$ onto $\X$, and its inverse is the $\lambda$-resolvent operator
\formula{
 U_\lambda f(x) & = f * u_\lambda(x) = \int_{\R^d} f(x + z) u_\lambda(z) dz ,
}
where
\formula{
 u_\lambda(z) & = \int_0^\infty e^{-\lambda t} p_t(z) dz , & \fourier u_\lambda(\xi) & = \frac{1}{\lambda + |\xi|^\alpha} \, .
}
Therefore, the fractional Laplace operator $L_S$, defined by~\eqref{eq:S} with domain $\dom_S(L, \X)$, has no essential extension $\tilde{L}$ on $\X$ such that $\lambda I - \tilde{L}$ is injective for some $\lambda > 0$; here $\X$ is any of $\leb^p$, $p \in [1, \infty)$, $\conto$ or $\contbu$. This is frequently used to prove equivalence of another definitions with the semigroup definition.

On $\conto$, injectivity of $\lambda I - \tilde{L}$ follows easily from the positive maximum principle: if $f \in \conto$ is in the domain of $\tilde{L}$ and $f(x) = \max \{ f(y) : y \in \R^d \}$, then $\tilde{L} f(x) \le 0$ (for complex-valued functions one requires that $f(x) = \max \{ |f(y)| : y \in \R^d \}$ implies $\re \tilde{L} f(x) \le 0$). Therefore, the fractional Laplace operator $L_S$, defined by~\eqref{eq:S} with domain $\dom_S(L, \conto)$, has no essential extension $\tilde{L}$ on $\conto$ which satisfies the positive maximum principle. It is probably well-known, although difficult to find in literature, that the above property extends to $\contbu$.

\begin{leftbar}
\begin{proposition}
\label{prop:pmp}
If $\tilde{L}$ is an extension of $L_S$, defined by~\eqref{eq:S} with domain $\dom(L_S, \contbu)$, and $\tilde{L}$ satisfies the positive maximum principle, then $\tilde{L}$ is in fact equal to $L_S$.
\end{proposition}
\end{leftbar}

\begin{proof}
Suppose that $\lambda > 0$, $f \in \contbu$, $\|f\|_\infty > 0$, and $\lambda f - \tilde{L} f = 0$. With no loss of generality we may assume that $\|f\|_\infty = 1$ and that in fact $\sup \{ f(x) : x \in \R^d \} = 1$. Fix $x$ such that $f(x) > \tfrac{1}{2}$. Let $g \in \dom(L_S, \conto)$ be such that $g(x) = \tfrac{1}{2}$, $\|g\|_\infty = \tfrac{1}{2}$ and $\|L_S g\|_\infty \le \tfrac{\lambda}{2}$. Since $f(x) + g(x) > 1$ and $\limsup_{|x| \to \infty} |f(x) + g(x)| \le 1$, the function $f + g$ attains a global maximum at some point $y$, and $f(y) + g(y) > 1$. Since $f + g$ is in the domain of $\tilde{L}$, by the positive maximum principle we have
\formula{
 0 & \ge \tilde{L}(f + g)(y) = \tilde{L} f(y) + L_S g(y) = \lambda f(y) + L_S g(y) \\
 & > \lambda (1 - g(y)) + L_S g(y) \ge \lambda(1 - \|g\|_\infty) - \|L_S g\|_\infty \ge \tfrac{\lambda}{2} - \tfrac{\lambda}{2} = 0 ,
}
a contradiction. Therefore, $\lambda I - \tilde{L}$ is injective, and hence $\tilde{L}$ is equal to $L$.
\end{proof}

Since $\fourier u_\lambda(\xi) = (\lambda + |\xi|^\alpha)^{-1} = \tfrac{1}{\lambda} \fourier u_1(\lambda^{-1/\alpha} \xi)$, we have
\formula{
 u_\lambda(x) & = \lambda^{(d - \alpha)/\alpha} u_1(\lambda^{1/\alpha} x) .
}
If $\alpha < d$, scaling~\eqref{eq:ptscaling} and the estimate~\eqref{eq:ptest} give
\formula{
 u_1(x) & \le C_1 \int_0^{|x|^\alpha} e^{-t} t |x|^{-d - \alpha} dt + C_1 \int_{|x|^\alpha}^\infty e^{-t} t^{-d/\alpha} dt \\
 & \le C_1 \min(\tfrac{1}{2} |x|^{-d + \alpha}, |x|^{-d - \alpha}) + C_1 \min(\tfrac{\alpha}{d - \alpha} |x|^{-d + \alpha}, \tfrac{2 \alpha}{d + \alpha} |x|^{-d - \alpha}) \\
 & \le C_2 \min(|x|^{-d + \alpha}, |x|^{-d - \alpha}) ,
}
where $C_1 = C_1(d, \alpha)$, $C_2 = C_2(d, \alpha)$; here we used the estimates $\int_0^a e^{-t} t \, dt \le \int_0^a t \, dt = \tfrac{1}{2} t^2$, $\int_0^a e^{-t} t \, dt \le \int_0^\infty e^{-t} t \, dt = 1$ for the former integral and $e^{-t} \le 1$ and $e^{-t} \le \tfrac{2}{t^2}$ for the latter integral. Similar estimates hold also when $\alpha > d$ and $\alpha = d = 1$; in the end, we have
\formula[eq:potest]{
 u_1(x) & \le C(d, \alpha) \begin{cases} \min(|x|^{-d + \alpha}, |x|^{-d - \alpha}) & \text{when $\alpha < d$,} \\ \min(\log (2 + \tfrac{1}{|x|}), |x|^{-2}) & \text{when $\alpha = d = 1$,} \\ \min(1, |x|^{-d - \alpha}) & \text{when $\alpha > d = 1$.} \end{cases}
}
In particular, $u_\lambda \in \leb^p$ if and only if $\tfrac{1}{p} > \tfrac{d - \alpha}{d}$, that is, if $p \in [1, \tfrac{d}{d - \alpha})$ when $\alpha < d$, $p \in [1, \infty)$ when $\alpha = d = 1$, and $p \in [1, \infty]$ when $\alpha > d$.

The above observation implies that if $f \in \dom(L_S, \leb^p)$, then $f \in \leb^q$ for every $q \in [p, (\tfrac{1}{p} - \tfrac{\alpha}{d})^{-1})$ when $p < \tfrac{d}{\alpha}$, for every $q \in [p, \infty)$ when $p = \tfrac{d}{\alpha}$, and for every $q \in [p, \infty]$ (and even $f \in \contbu$) when $p > \tfrac{d}{\alpha}$. In the case $p < \tfrac{d}{\alpha}$ in fact $f \in \leb^q$ also for $q = (\tfrac{1}{p} - \tfrac{\alpha}{d})^{-1}$ by the weak Young's inequality, see Section~\ref{sec:R} for more details.

For a detailed treatment of semigroups of operators on Banach spaces, see, for example,~\cite{bib:ds53}. Fractional powers of generators of semigroups are studied in detail in~\cite{bib:ms01}. For further properties of the space $\dom(L_S, \leb^p)$, traditionally denoted by $\leb^p_\alpha$, see~\cite[Chapter~7]{bib:s01} and~\cite[Chapter~V]{bib:s70}. In particular, $\dom(L_S, \leb^p)$ coincides with the space of Bessel potentials of $\leb^p$ functions, see~\cite[Theorem~7.16]{bib:s01}.

%
%

\subsection{Riesz potentials}
\label{sec:R}

Suppose that $\alpha < d$, that is, $d \ge 2$ and $\alpha \in (0, 2)$ or $d = 1$ and $\alpha \in (0, 1)$. Then the function $|\xi|^{-\alpha}$ is integrable in the unit ball and bounded in the complement of the unit ball, and hence it is the Fourier transform of a tempered distribution. Using the identity
\formula{
 \lambda^{-\alpha/2} & = \frac{1}{\Gamma(\tfrac{\alpha}{2})} \int_0^\infty e^{-t \lambda} t^{-1 + \alpha/2} dt
}
and following the argument used in Section~\ref{sec:B} and in~\eqref{eq:nucalc}, one easily shows that, at least in the sense of spectral theory on $\leb^2$,
\formula{
 (-L)^{-1} f(x) & = c_{d,-\alpha} \int_{\R^d} f(x + z) |z|^{-d + \alpha} dz ,
}
with $c_{d,-\alpha}$ defined as in~\eqref{eq:cda}.

The same expression can be obtained using the general theory of strongly continuous semigroups, which tells that, at least formally, $(-L)^{-1}$ is the $0$-resolvent operator
\formula{
 U_0 f & = \int_0^\infty P_t f dt .
}
Indeed, using the Bochner's subordination formula~\eqref{eq:ptbochner}, one easily shows that
\formula{
 \int_0^\infty p_t(z) dt & = \frac{1}{\Gamma(\tfrac{\alpha}{2})} \int_0^\infty t^{-1 + \alpha/2} k_t(z) dt = c_{d,-\alpha} |z|^{-d + \alpha} ,
}
see, for example, \cite{bib:bbkrsv09} for further details. We remark that under appropriate conditions, $U_0$ is the inverse of $-L_S$ in the general context of generators of semigroups of contractions on Banach spaces. To be specific, let $P_t$ form such a semigroup on $\X$, and suppose that $P_t f$ converges to $0$ as $t \to \infty$ for all $f \in \X$ (in our case this is true when $\X = \leb^p$, with $p \in [1, \infty)$, or $\X = \conto$). If $f \in \dom(L_S, \X)$, then
\formula{
 \lim_{t \to \infty} \int_0^t P_s L_S f ds & = \lim_{t \to \infty} (P_t f - f) = -f
}
(see Section~\ref{sec:process} for further details). Hence $U_0 L_S f = -f$ whenever the norm of $P_t L_S f$ is integrable with respect to $t > 0$. Conversely, if $g \in \X$ and the norm of $P_t g$ is an integrable function of $t$, then
\formula[eq:li]{
 L_S I_\alpha g & = \lim_{t \to 0^+} \frac{1}{t} \expr{\int_t^\infty - \int_0^\infty} P_s g ds = -\lim_{t \to 0^+} \frac{1}{t} \int_0^t P_s g ds = -g .
}
A pointwise version of the above identity is also true; see Proposition~\ref{prop:pot:gen} for a slightly more general statement.

The above considerations motivate the following definition.

\begin{leftbar}
\begin{definition}[Riesz potential definition of $L$]
If $\alpha < d$, then the fractional Laplace operator is the inverse of the Riesz potential, namely $L_R f = -g$ whenever
\formula{
\tag{R}\label{eq:R}
 f(x) & = I_\alpha g(x) = c_{d,-\alpha} \int_{\R^d} g(x + z) |z|^{-d + \alpha} dz ,
}
with $c_{d,-\alpha}$ defined as in~\eqref{eq:cda}. More precisely, if $f, g \in \X$, the integral in~\eqref{eq:R} is finite and the equality therein holds for all $x \in \R^d$ (when $\X = \conto$) or for almost all $x \in \R^d$ (when $\X = \leb^p$ with $p \in [1, \infty]$), then we write $f \in \dom(L_R, \X)$.
\end{definition}
\end{leftbar}

Note that if $f \in \leb^p$ and $p \in [1, \tfrac{d}{\alpha})$, then the convolution $I_\alpha f$ of $f$ and $|z|^{-d + \alpha}$ is in $\leb^p + \leb^\infty$ (because the convolution with $|z|^{-d + \alpha} \ind_B(z)$ is in $\leb^p$, while the convolution with $|z|^{-d + \alpha} \ind_{\R^d \setminus B}$ is in $\leb^\infty$). Using weak Young's inequality, one proves that in fact $I_\alpha$ continuously maps $\leb^p$ into $\leb^q$ if (and only if) $\tfrac{1}{q} = \tfrac{1}{p} - \tfrac{\alpha}{d}$, $p \in (1, \infty)$, see~\cite[Theorem~1 in Chapter~V]{bib:s70}. Therefore, if $p < \tfrac{d}{\alpha}$, $q = (\tfrac{1}{p} - \tfrac{\alpha}{d})^{-1}$ and $f \in \dom(L_R, \leb^p)$, then $f \in \leb^p \cap \leb^q$. For this reason it is sometimes more convenient to consider $I_\alpha$ as an operator from $\leb^p$ to $\leb^q$, and its inverse $L_R$ as an operator from $\leb^q$ to $\leb^p$; see~[Section~7.3]\cite{bib:s01} for details.

For $p \in (\tfrac{d}{\alpha}, \infty]$, $I_\alpha f$ can be defined in the appropriate space of distributions, see~\cite[Section~I.1]{bib:l72} and~\cite[Section~7.1]{bib:s01} for more details. We remark that this extension can give a similar description of the domain $\dom(L_R, \leb^p)$ as in Theorem~\ref{th:main} for general $p \in [1, \infty)$, see~\cite[Theorem~7.18]{bib:s01}.

For the detailed analysis of Riesz potential operators $I_\alpha$, we refer to~\cite{bib:l72,bib:r38a,bib:r38b,bib:s01,bib:s70}.

%
%

\subsection{Harmonic extensions}
\label{sec:H}

For $\lambda \ge 0$, there is exactly one solution $\ph_\lambda$ of the second order ordinary differential equation:
\formula[eq:ode]{
 \alpha^2 c_\alpha^{2 / \alpha} y^{2 - 2/\alpha} \partial_y^2 \ph_\lambda(y) & = \lambda \ph_\lambda(y)
}
which is non-negative, continuous and bounded on $[0, \infty)$, and which satisfies $\ph_\lambda(0) = 1$. Here $c_\alpha = 2^{-\alpha} |\Gamma(-\tfrac{\alpha}{2})| / \Gamma(\tfrac{\alpha}{2})$. This solution is given by
\formula{
 \ph_\lambda(y) & = \frac{2^{1-\alpha/2}}{\Gamma(\tfrac{\alpha}{2})} \expr{\frac{\lambda^{\alpha/2} y}{c_\alpha}}^{1/2} K_{\alpha/2}\expr{\expr{\frac{\lambda^{\alpha/2} y}{c_\alpha}}^{1/\alpha}} ,
}
where $K_{\alpha/2}$ is the modified Bessel function of the second kind (by continuity, we let $\ph_\lambda(0) = \ph_0(y) = 1$).  By a simple calculation, we have $\ph_\lambda'(0) = -\lambda^{\alpha/2}$. This property is crucial for the extension technique, described below.

Before we proceed, we note that for $\alpha = 1$ the equation is simply \smash{$\partial_y^2 \ph_\lambda(y) = \lambda \ph_\lambda(y)$}, and \smash{$\ph_\lambda(y) = \exp(\sqrt{\lambda} \, y)$}. Note also that the solution of~\eqref{eq:ode} linearly independent from $\ph_\lambda$ is given by a similar formula, with $K_{\alpha/2}$ replaced by the modified Bessel function of the first kind $I_{\alpha/2}$.

Let $f \in \leb^2$ and consider the following partial differential equation:
\formula{
 \begin{cases}
 \Delta_x u(x, y) + \alpha^2 c_\alpha^{2 / \alpha} y^{2 - 2/\alpha} \partial_y^2 u(x, y) = 0 \qquad & \text{for $y > 0$,} \\
 u(x, 0) = f(x) ,
 \end{cases}
}
together with the following regularity conditions: $u(x, y)$, as a function of $x \in \R^d$, is in $\leb^2$ for each $y \in [0, \infty)$, with norm bounded uniformly in $y \in [0, \infty)$; and $u(x, y)$ depends continuously on $y \in [0, \infty)$ with respect to the $\leb^2$ norm. Then for each $\xi \in \R^d$ the Fourier transform $\fourier u(\xi, y)$ of $u(x, y)$ with respect to the $x$ variable is bounded in $y \in [0, \infty)$, satisfies~\eqref{eq:ode} with $\lambda = |\xi|^2$, and is equal to $\fourier f(\xi)$ for $y = 0$. It follows that
\formula[eq:fourier:u]{
 \fourier u(\xi, y) & = \ph_{|\xi|^2}(y) \fourier f(\xi) ,
}
and therefore
\formula{
 \partial_y \fourier u(\xi, 0) & = \ph_{|\xi|^2}'(0) \fourier f(\xi) = -|\xi|^\alpha \fourier f(\xi) ,
}
that is, $\partial_y u(x, 0)$ is equal to $L_F f(x)$, the fractional Laplace operator applied to $f$, defined using Fourier transform.

The same method applies not only to $\leb^2$. The fractional Laplace operator $L$ on a Banach space $\X$, defined using harmonic extensions, is the Dirichlet-to-Neumann operator for the weighted Dirichlet problem in the half-space:
\formula[eq:pde]{
 \begin{cases}
 \Delta_x u(x, y) + \alpha^2 c_\alpha^{2 / \alpha} y^{2 - 2/\alpha} \partial_y^2 u(x, y) = 0 \qquad & \text{for $y > 0$,} \\
 u(x, 0) = f(x) , \\
 \partial_y u(x, 0) = L_H f(x) ,
 \end{cases}
}
where
\formula[eq:ca]{
 c_\alpha & = \frac{|\Gamma(-\tfrac{\alpha}{2})|}{2^\alpha \Gamma(\tfrac{\alpha}{2})} \, .
}
The problem~\eqref{eq:pde} requires a regularity condition on $u$, which asserts that the (distributional) Fourier transform of $u$ has the desired form~\eqref{eq:fourier:u}. We state without a proof that if $\X$ is one of the spaces $\leb^p$, $p \in [1, \infty]$, $\conto$ or $\contbu$, it is sufficient to assume that $u(x, y)$, as a function of $x \in \R^d$, is in $\X$ for each $y \in [0, \infty)$, with norm bounded uniformly in $y \in [0, \infty)$, and that $u(x, y)$ depends continuously on $y \in [0, \infty)$ with respect to the norm of $\X$.

It is convenient to rephrase the definition~\eqref{eq:pde} in the following way. The property~\eqref{eq:fourier:u} of the distributional Fourier transform is equivalent to the condition $u(x, y) = f * q_y(x)$, where
\formula[eq:fqy]{
 \fourier q_y(\xi) & = \ph_{|\xi|^2}(y) = \frac{2^{1-\alpha/2}}{\Gamma(\tfrac{\alpha}{2})} \expr{\frac{|\xi|^\alpha y}{c_\alpha}}^{1/2} K_{\alpha/2}\expr{\expr{\frac{|\xi|^\alpha y}{c_\alpha}}^{1/\alpha}} .
}
By~\cite[formula~6.565.4]{bib:gr07}, after simplification we obtain a surprisingly elementary expression $q_y(z) = c_{d,\alpha} y ((y / c_\alpha)^{2/\alpha} + |z|^2)^{-(d + \alpha)/2}$, with $c_{d,\alpha}$ given by~\eqref{eq:cda} (see~\cite[Section~2.4]{bib:cs07} for an alternative derivation). Now the definition can be given in the same way as in~\eqref{eq:S}, using the kernel $q_y$ instead of $p_t$.

For $\alpha = 1$, $q_y(z) = p_y(z)$ is the Poisson kernel for the half-space in $\R^{d+1}$. For general $\alpha \in (0, 2)$, $q_y(z)$ has similar properties to $p_t(z)$: it is strictly positive and infinitely smooth. Furthermore, $q_t(z) = t^{d/\alpha} q_1(t^{1/\alpha} z)$, and a version of~\eqref{eq:heatlim} holds for $q_y$.

\begin{leftbar}
\begin{definition}[harmonic extension definition of $L$]
The fractional Laplace operator is given by
\formula{
\tag{H}\label{eq:H}
 \begin{aligned}
 L_H f(x) & = \lim_{y \to 0^+} \frac{f * q_y(x) - f(x)}{y} \\
 & = \lim_{y \to 0^+} \int_{\R^d} (f(x + z) - f(x)) \, \frac{q_y(z)}{y} \, dz ,
 \end{aligned}
}
where
\formula[eq:qy]{
 q_y(z) & = c_{d,\alpha} \, \frac{y}{((y / c_\alpha)^{2/\alpha} + |z|^2)^{(d + \alpha)/2}} \, ,
}
$c_{d,\alpha}$ is given by~\eqref{eq:cda} and $c_\alpha$ is given by~\eqref{eq:ca}. We write $f \in \dom(L_H, x)$ if the limit exists for a given $x \in \R^d$. If $f \in \X$ and the limit exists in $\X$, then we write $f \in \dom(L_H, \X)$.
\end{definition}
\end{leftbar}

The idea of linking the fractional Laplace operator with a Dirichlet-to-Neumann operator for $\alpha = 1$ dates back at least to the work of Spitzer~\cite{bib:s58}. It was heavily used, for example, in hydrodynamics~\cite{bib:fk83,bib:fl48,bib:h64,bib:kk04}, and, more recently, in probability theory~\cite{bib:bk04,bib:bk06,bib:bbkrsv09,bib:kkms10}. Similar relation for general $\alpha$ appeared first in the article by Molchanov and Ostrovski~\cite{bib:mo69}, and later in~\cite{bib:d90,bib:d04,bib:dm07}. It became well-recognised in analysis as the Caffarelli--Silvestre extension technique after it was rediscovered in~\cite{bib:cs07}, see, for example, \cite{bib:css08}. Fractional powers of more general operators on arbitrary Banach spaces can be studied in a similar way, see~\cite{bib:gms13,bib:st10}. The most general version of the above approach is to consider an arbitrary operator \smash{$\tilde{L}$} in the $x$ variable, and an arbitrary elliptic differential operator $a(y) \partial_y^2$ in the $y$ variable. When \smash{$\tilde{L}$} acts on some Hilbert space, spectral theory for operators on Hilbert spaces and Kre\u{\i}n's spectral theory of strings (see~\cite[Chapter~15]{bib:ssv12} and the references therein) imply that the resulting operator $L$ is equal to $-\psi(-\tilde{L})$ for some \emph{operator monotone function}, or \emph{complete Bernstein function}, $\psi$, and there is a one-to-one correspondence between such $\psi$ and the coefficient $a(y)$ (as long as one agrees for certain singularities of $a(y)$). In many cases this approach extends to other function spaces, such as $\leb^p$ and $\conto$.

Various variants of~\eqref{eq:pde} can be obtained by a change of variable. A divergence form follows by taking $y = c_\alpha z^\alpha$:
\formula[eq:pde1]{
 \begin{cases}
 \nabla_{x,z} \cdot (z^{1 - \alpha} \nabla_{x,z} \, u)(x, z) = 0 \qquad & \text{for $z > 0$,} \\
 u(x, 0) = f(x) , \\
 \lim\limits_{z \to 0^+} \dfrac{u(x, z) - u(x, 0)}{c_\alpha z^\alpha} = L_H f(x) .
 \end{cases}
}
Another interesting variant is obtained by expanding the partial derivatives:
\formula[eq:pde2]{
 \begin{cases}
 \Delta_{x,z} u(x, z) + \tfrac{1 - \alpha}{z} \partial_z u(x, z) = 0 \qquad & \text{for $z > 0$,} \\
 u(x, 0) = f(x) , \\
 \lim\limits_{z \to 0^+} \dfrac{u(x, z) - u(x, 0)}{c_\alpha z^\alpha} = L_H f(x) .
 \end{cases}
}

%
%

\section{Pointwise convergence}
\label{sec:pointwise}

In this section we discuss relationship between pointwise definitions~\eqref{eq:B}, \eqref{eq:I}, \eqref{eq:II}, \eqref{eq:III}, \eqref{eq:D}, \eqref{eq:S} and~\eqref{eq:H} for a given $x$. Some of these results depend on explicit expressions for $L f(x)$ and thus extensions to more general operators are problematic.

Our first two results are standard.

\begin{leftbar}
\begin{lemma}
\label{lem:ppip}
If $f \in \dom(L_\st{I}, x)$, then $f \in \dom(L_I, x)$, and $L_\st{I} f(x) = L_I f(x)$. Similarly, if $f \in \dom(L_\stt{I}, x)$, then $f \in \dom(L_I, x)$, and $L_\stt{I} f(x) = L_I f(x)$.
\end{lemma}
\end{leftbar}

\begin{proof}
For the first statement, it suffices to observe that
\formula{
 \int_{\R^d \setminus B_r} (f(x + z) - f(x) - \nabla f(x) \cdot z \ind_B(z)) \nu(z) dz & = \int_{\R^d \setminus B_r} (f(x + z) - f(x)) \nu(z) dz
}
and take a limit as $r \to 0^+$. The other one is proved in a similar way.
\end{proof}

Note that any bounded $f$ such that $f(x + z) = -f(x - z)$ belongs to $\dom(L_I, x)$, but not every such function has a gradient at $x$, so $\dom(L_\st{I}, x)$ is a proper subset of $\dom(L_I, x)$. In a similar way, it is easy to construct a bounded function $f$ such that the integral of $\ph(z) = (f(x + z) + f(x - z) - 2 f(x)) / |z|^{d + \alpha}$ over $\R^d \setminus B_r$ has a finite limit as $r \to 0^+$, but a similar limit for the integral of $|\ph(z)|$ is infinite. This shows that inclusions in the above lemma are proper.

\begin{leftbar}
\begin{lemma}
\label{lem:smooth}
If $f$ has second order partial derivatives at $x$ and $(1 + |z|)^{-d - \alpha} f(z)$ is integrable, then $f$ is in $\dom(L_D, x)$, $\dom(L_I, x)$, $\dom(L_\st{I}, x)$, $\dom(L_\stt{I}, x)$, $\dom(L_S, x)$, $\dom(L_H, x)$ and $\dom(L_B, x)$, and all definitions of $L f(x)$ agree. If $f$ is of class $\cont^2$ in $B(x, r)$, then the rates of convergence in each of the definitions~\eqref{eq:D}, \eqref{eq:I}, \eqref{eq:S} and \eqref{eq:H} depend only on $r$, $\sup \{\max(|f(y)|, |\nabla f(y)|, |\nabla^2 f(y)|) : y \in B(x, r)\}$ and $\int_{\R^d} (1 + |z|)^{-d - \alpha} |f(x + z)| dz$.
\end{lemma}
\end{leftbar}

\begin{proof}
The result follows easily from Taylor's expansion of $f$ at $x$, namely
\formula{
 f(x + z) = f(x) + z \cdot \nabla f(x) + O(|z|^2) ,
}
and the symmetry and upper bounds for the appropriate convolution kernels. We omit the details.
\end{proof}

The following result seems to be new.

\begin{leftbar}
\begin{lemma}
\label{lem:dip}
If $f \in \dom(L_D, x)$, then $f \in \dom(L_I, x)$, and $L_D f(x) = L_I f(x)$.
\end{lemma}
\end{leftbar}

\begin{proof}
Let $0 < s < t$. Substituting $r^2 = s^2 + v (t^2 - s^2)$ and using~\cite[formula~3.197.4]{bib:gr07}, we obtain
\formula{
 & \int_s^t \frac{1}{t^d (t^2 - r^2)^{\alpha/2}} \, \frac{1}{r (r^2 - s^2)^{1 - \alpha/2}} \, dr \\
 & \qquad = \frac{1}{2 s^2 t^d} \int_0^1 \frac{v^{\alpha/2 - 1} (1 - v)^{-\alpha/2}}{(1 + v (t^2/s^2 - 1))} \, dv \\
 & \qquad = \frac{\Gamma(\tfrac{\alpha}{2}) \Gamma(1 - \tfrac{\alpha}{2})}{2 s^2 t^d (t^2/s^2)^{\alpha/2}} = \frac{\alpha \, \Gamma(\tfrac{\alpha}{2}) |\Gamma(-\tfrac{\alpha}{2})|}{4 s^{2 - \alpha} t^{d + \alpha}} \, .
}
Taking $s = R$, $t = |z|$ and multiplying both sides by $c_{d,\alpha}$ leads to
\formula[eq:numu]{
 \nu_R(z) & = \frac{4 R^{2 - \alpha}}{\alpha \, \Gamma(\tfrac{\alpha}{2}) |\Gamma(-\tfrac{\alpha}{2})|} \int_R^\infty \tilde{\nu}_r(z) \, \frac{1}{r (r^2 - R^2)^{1 - \alpha/2}} \, dr .
}
Note that both sides of the above equality are zero in $\overline{B}_R$, and so it extends to all $R > 0$ and all $z$.

The remaining part of the proof is standard. Let $f \in \dom(L_D, x)$ and
\formula{
 \ph(r) & = \int_{\R^d} (f(x + z) - f(x)) \tilde{\nu}_r(z) dz
}
for $r > 0$. Then $\ph$ converges to a limit $\ph(0^+) = L_D f(x)$ as $r \to 0^+$. In particular, $\ph$ is bounded on some interval $(0, r_0)$. Since $\nu_r(z) \le \tilde{\nu}_r(z)$, the definition~\eqref{eq:D} requires that $(f(x + z) - f(x)) \nu_r(z)$ is absolutely integrable for all $r > 0$. Hence, by~\eqref{eq:numu},
\formula[eq:dipaux]{
\begin{aligned}
 \int_{\R^d} (f(x + z) - f(x)) \nu_R(z) dz & = \frac{4 R^{2 - \alpha}}{\alpha \, \Gamma(\tfrac{\alpha}{2}) |\Gamma(-\tfrac{\alpha}{2})|} \int_R^\infty \ph(r) \, \frac{1}{r (r^2 - R^2)^{1 - \alpha/2}} \, dr \\
 & = \frac{4}{\alpha \, \Gamma(\tfrac{\alpha}{2}) |\Gamma(-\tfrac{\alpha}{2})|} \int_1^{r_0 / R} \ph(R s) \, \frac{1}{s (s^2 - 1)^{1 - \alpha/2}} \, ds \\
 & \qquad + \frac{4}{\alpha \, \Gamma(\tfrac{\alpha}{2}) |\Gamma(-\tfrac{\alpha}{2})|} \int_{r_0}^\infty \ph(r) \, \frac{1}{r ((\tfrac{r}{R})^2 - 1)^{1 - \alpha/2}} \, dr,
\end{aligned}
}
and all integrals above are absolutely convergent. We consider the two integrals in the right-hand side of~\eqref{eq:dipaux} separately. By dominated convergence, as $R \to 0^+$, the former one converges to (see~\cite[formula~3.191.2]{bib:gr07})
\formula{
 -\frac{4 \ph(0^+)}{\alpha \, \Gamma(\tfrac{\alpha}{2}) \Gamma(-\tfrac{\alpha}{2})} \int_1^\infty \frac{1}{s (s^2 - 1)^{1 - \alpha/2}} \, ds = \ph(0^+) .
}
In the latter integral, the absolute value of the integrand decreases to $0$ as $R \to 0^+$, and the integral is absolutely integrable for all $R \in (0, r_0)$. Again by dominated convergence, the latter integral in the right-hand side of~\eqref{eq:dipaux} converges to $0$ as $R \to 0^+$.
\end{proof}


It is easy to construct $f \in \dom(L_I, x)$ which is not in $\dom(L_D, x)$. 

\begin{example}
Let
\formula{
 f(x + z) & = \sum_{n = 1}^\infty \frac{\eps_n}{(|z|^2 - r_n^2)^{1 - \alpha/2}} \, \ind_{[r_n, 2 r_n]}(|z|) ,
}
where, for example, $r_n = 2^{-n}$ and $\eps_n = 5^{-n}$. Then
\formula{
 \int_{\R^d} f(x + z) \nu(z) dz & = \sum_{n = 1}^\infty \frac{c_{d,\alpha} \eps_n}{r_n^2} \int_{B_2 \setminus B_1} \frac{1}{|y|^{d + \alpha} (|y|^2 - 1)^{1 - \alpha/2}} \, dy < \infty ,
}
so $f \in \dom(L_I, x)$, but
\formula{
 \int_{\R^d} f(x + z) \tilde{\nu}_{r_n}(z) dz & = \infty ,
}
so that $f \notin \dom(L_D, x)$.
\end{example}

The next result is likely well-known, although the author could not find it in the literature. Clearly, if $f(x + z) = -f(x - z)$ for $z \in \R^d$ and $f(x + z) |z|^{d + \alpha}$ is integrable in $\R^d \setminus B(x, r)$ for every $r > 0$, then $f \in \dom(L_I, x)$ and $L_I f(x) = 0$. Nevertheless, $f$ may fail to be locally integrable near $x$, and so $f \notin \dom(L_S, x)$. For locally integrable $f$, however, the pointwise definition of $L_S$ is indeed an extension of the pointwise definition of $L_I$.

\begin{leftbar}
\begin{lemma}
\label{lem:pis}
If $f \in \dom(L_I, x)$ and $f$ is locally integrable near $x$, then $f \in \dom(L_S, x)$, and $L_I f(x) = L_S f(x)$.
\end{lemma}
\end{leftbar}

\begin{proof}
Let $m(r)$ be the profile function of $c_{d,\alpha}^{-1} |z|^{d + \alpha} p_1(z)$. Later in this proof we show that $|m'(r)|$ is integrable on $(0, \infty)$. Once this is done, the argument is very similar to the one used in the proof of Lemma~\ref{lem:dip}. Observe that $m(0) = 0$ and $m(r) \to 1$ as $r \to \infty$, so that the integral of $m'(r)$ is equal to $1$. Since
\formula{
 c_{d,\alpha}^{-1} |z|^{d + \alpha} p_t(z) & = c_{d,\alpha}^{-1} t^{-d/\alpha} |z|^{d + \alpha} p_1(t^{-1/\alpha} z) = t m(t^{-1/\alpha} |z|) \\
 & = t \int_0^{t^{-1/\alpha} |z|} m'(r) dr = t \int_0^\infty \ind_{\R^d \setminus B_r}(t^{-1/\alpha} z) m'(r) dr ,
}
we have, by Fubini,
\formula{
 p_t(z) & = t \int_0^\infty c_{d,\alpha} |z|^{-d - \alpha} \ind_{\R^d \setminus B_r}(t^{-1/\alpha} z) m'(r) dr \\
 & = t \int_0^\infty \nu_{t^{1/\alpha} r}(z) m'(r) dr .
}
Let $f \in \dom(L_I, x)$ and
\formula{
 \ph(r) & = \int_{\R^d} (f(x + z) - f(x)) \nu_r(z) dz
}
for $r > 0$. Then $\ph$ is bounded and it converges to a limit $\ph(0^+) = L_I f(x)$ as $r \to 0^+$. Since
\formula{
 \int_{\R^d} (f(x + z) - f(x)) \, \frac{p_t(z)}{t} \, dz & = \int_0^\infty \ph(t^{1/\alpha} r) m'(r) dr ,
}
the desired result follows by dominated convergence.

It remains to prove integrability of $|m'(r)|$. Observe that
\formula{
 m'(|z|) & = c_{d,\alpha}^{-1} |z|^{d + \alpha - 1} ((d + \alpha) p_1(z) + z \cdot \nabla p_1(z)) ,
}
so that the Fourier transform of $g(z) = c_{d,\alpha} |z|^{1 - d - \alpha} m'(|z|)$ is equal to
\formula{
 \fourier g(\xi) & = (d + \alpha) e^{-|\xi|^\alpha} - \nabla \cdot (\xi e^{-|\xi|^\alpha}) \\
 & = (d + \alpha) e^{-|\xi|^\alpha} - (d - \alpha |\xi|^\alpha) e^{-|\xi|^\alpha} \\
 & = \alpha e^{-|\xi|^\alpha} (1 + |\xi|^\alpha) .
}
The right-hand side is smooth in $\R^d \setminus \{0\}$, and its series expansion at $0$ is $\fourier g(\xi) = \alpha (1 - \tfrac{1}{2} |\xi|^{2 \alpha}) + O(|\xi|^{3 \alpha})$. By~\cite[Theorem~4]{bib:l83},
\formula{
 \lim_{r \to \infty} c_{d,\alpha} r^{1 + \alpha} m'(r) & = \lim_{|z| \to \infty} |z|^{d + 2 \alpha} g(z) = \frac{2^{2 \alpha - d/2 - 1} \Gamma(\tfrac{d}{2} + \alpha)}{\Gamma(-\alpha)} \, ,
}
where the right-hand side is understood to be equal to $0$ if $\alpha = 1$. It follows that for $\alpha \in (0, 1)$, $m'(r)$ is ultimately positive, while for $\alpha \in (1, 2)$, $m'(r)$ is ultimately negative. By~\eqref{eq:cauchy}, for $\alpha = 1$, $m'(r)$ is everywhere positive (ultimate positivity also follows from~\cite[Theorem~4]{bib:l83} by considering the next term in the series expansion of $\fourier g(\xi)$). Since $m'(r)$ is smooth on $[0, \infty)$ and
\formula{
 \lim_{R \to \infty} \int_0^R m'(r) dr = \lim_{R \to \infty} m(R) - m(0) = \lim_{|z| \to \infty} c_{d,\alpha}^{-1} |z|^{d + \alpha} p_1(z) = 1 ,
}
we conclude that $|m'(r)|$ is integrable, as desired.
\end{proof}

In a similar way, we obtain the following result.

\begin{leftbar}
\begin{lemma}
\label{lem:pih}
If $f \in \dom(L_I, x)$ and $f$ is locally integrable near $x$, then $f \in \dom(L_H, x)$, and $L_I f(x) = L_H f(x)$.
\end{lemma}
\end{leftbar}

\begin{proof}
The argument is exactly the same as the proof of Lemma~\ref{lem:pis}, once we note that by~\eqref{eq:qy}, the profile function of $c_{d,\alpha}^{-1} |z|^{d + \alpha} q_1(z)$, namely $m(r) = (r^2 / (c_\alpha^{-2/\alpha} + r^2))^{(d + \alpha)/2}$, is increasing.
\end{proof}

In order to prove a similar result connecting~\eqref{eq:H} and~\eqref{eq:S}, one would need some relationship between $q_y(z)$ and $p_t(z)$. The statement is trivial for $\alpha = 1$, because then $q_y(z) = p_y(z)$. For $\alpha \in (1, 2)$, we conjecture that $\ph(r) = \sqrt{r} \, K_{\alpha/2}(r^{1/\alpha})$ is completely monotone. If this is true, $\fourier q_y(\xi)$ can be expressed as the integral average of $\fourier p_t(\xi) = e^{-t |\xi|^\alpha}$, and so $f \in \dom(L_S, x)$ implies $f \in \dom(L_H, x)$. For $\alpha \in (0, 1)$ it is unclear whether $\fourier p_t(\xi)$ can be expressed as the integral average of $\fourier q_y(\xi)$.

\begin{leftbar}
\begin{conjecture}
For $\alpha \in (1, 2)$ the function $\sqrt{r} \, K_{\alpha/2}(r^{1/\alpha})$ is completely monotone in $r \in (0, \infty)$.
\end{conjecture}
\end{leftbar}

As in Lemma~\ref{lem:dip}, the inclusions in Lemmas~\ref{lem:pis} and~\ref{lem:pih} are proper. An example of $f \in \dom(L_S, x)$ which is not in $\dom(L_I, x)$ is, however, more complicated, and we only sketch the argument.

\begin{example}
Let
\formula{
 f(x + z) & = \sum_{n = 1}^\infty \eps_n |z|^{1 + \alpha} (\ind_{[r_n, (1 + \delta_n) r_n]}(|z|) - \ind_{[(1 - \delta_n) r_n, r_n]}(|z|)) ,
}
where, for example, $r_n = 2^{-n}$, $\delta_n = 4^{-n}$ and $\eps_n = n 8^n$. First of all, $f$ is easily proved to be integrable. Due to cancellations, the integral of $f(x + z) |z|^{-d - \alpha}$ over $z \in B_{(1 + \delta_n) r_n} \setminus B_{(1 - \delta_n) r_n}$ is zero, and therefore
\formula{
 \int_{\R^d} f(x + z) \nu_{r_n}(z) dz & = c_{d,\alpha} \int_{B_{(1 + \delta_n) r_n} \setminus B_{r_n}} f(x + z) |z|^{-d - \alpha} dz = c_{d,\alpha} d |B| \eps_n \delta_n r_n \to \infty .
}
It follows that $f \notin \dom(L_I, x)$. However, if $m(r)$ denotes the profile function of $|z|^{d + \alpha} p_1(z)$ and $M$ is the supremum of $|m'(r)| (1 + r^{1 + \alpha})$ (which was shown to be finite in the proof of Lemma~\ref{lem:pis}), then
\formula{
 \abs{\int_{\R^d} f(x + z) \, \frac{p_t(z)}{t} \, dz} & \le \sum_{n = 1}^\infty \eps_n \abs{\expr{\int_{B_{(1 + \delta_n) r_n} \setminus B_{r_n}} - \int_{B_{r_n} \setminus B_{(1 - \delta_n) r_n}}} |z|^{1 + \alpha} \, \frac{p_t(z)}{t} \, dz} \\
 & = \sum_{n = 1}^\infty d |B| \eps_n \abs{\expr{\int_{r_n}^{(1 + \delta_n) r_n} - \int_{(1 - \delta_n) r_n}^{r_n}} m(t^{-1/\alpha} r) dr} \displaybreak[0] \\
 & = \sum_{n = 1}^\infty d |B| \eps_n r_n \abs{\int_0^{\delta_n} (m(t^{-1/\alpha} (1 + s) r_n) - m(t^{-1/\alpha} (1 - s) r_n)) ds} \displaybreak[0] \\
 & \le \sum_{n = 1}^\infty d |B| \eps_n r_n \, \frac{t^{-1/\alpha} r_n M \delta_n^2}{1 + (\tfrac{1}{2} t^{-1/\alpha} r_n)^{1 + \alpha}} \\
 & = \sum_{n = 1}^\infty 2^{1 + \alpha} d |B| M \eps_n \delta_n^2 r_n^{1 - \alpha} \, \frac{t}{(2 t^{1/\alpha} r_n^{-1})^{1 + \alpha} + 1} \to 0
}
as $t \to 0^+$. This proves that $f \in \dom(L_S, x)$, and a similar argument shows that $f \in \dom(L_H, x)$. We omit the details.
\end{example}

Apparently an example can be given to prove that $f \in \dom(L_S, x)$ does not imply $f \in \dom(L_B, x)$, but the author could not work out the technical details. On the other hand, it is not true that $f \in \dom(L_B, x)$ implies $f \in \dom(L_S, x)$: pointwise convergence in~\eqref{eq:B} does not require integrability of $(1 + |z|)^{-d - \alpha} f(z)$ at infinity. For example, when $d \ge 2$, the function $f(x) = x_1^2 - x_2^2$ can be proved to belong to $\dom(L_B, x)$ (and in fact $L_B f$ is everywhere zero according to~\eqref{eq:B}), but due to fast growth of $f$ at infinity, $f$ does not belong to $\dom(L_S, x)$. This is, however, the only obstacle in the proof of the following result.

\begin{leftbar}
\begin{lemma}
\label{lem:bis}
If $f \in \dom(L_B, x)$ and $(1 + |z|)^{-d - \alpha} f(z)$ is integrable, then $f \in \dom(L_S, x)$ and $f \in \dom(L_H, x)$, and $L_B f(x) = L_S f(x) = L_H f(x)$.
\end{lemma}
\end{leftbar}

\begin{proof}
Recall that by the Bochner's subordination formula~\eqref{eq:ptbochner},
\formula{
 p_t(z) & = t^{-2/\alpha} \int_0^\infty k_r(z) \eta(t^{-2/\alpha} r) dr ,
}
where $k_r(z)$ is the Gauss--Weierstrass kernel and $\eta(s)$ is a smooth function such that $0 < \eta(s) \le C_\alpha \min(1, s^{-1 - \alpha/2})$ for $s > 0$ and $\lim_{s \to \infty} s^{1 + \alpha/2} \eta(s) = 1 / |\Gamma(-\tfrac{\alpha}{2})|$ (see~\cite[Remark~14.18]{bib:s99}). Therefore, by Fubini,
\formula{
 \int_{\R^d} (f(x + z) - f(x)) \, \frac{p_t(z)}{t} \, dz & = \int_0^\infty (k_r * f(x) - f(x)) t^{-1 - 2/\alpha} \eta(t^{-2/\alpha} r) dr .
}
Suppose that $f \in \dom(L_B, x)$. Then, by dominated convergence,
\formula{
 \lim_{t \to 0^+} \int_{\R^d} (f(x + z) - f(x)) \, \frac{p_t(z)}{t} \, dz & = \frac{1}{|\Gamma(-\tfrac{\alpha}{2})|} \int_0^\infty (k_r * f(x) - f(x)) r^{-1 - \alpha/2} dr ,
}
as desired.

A similar argument involving the identity
\formula{
 q_y(z) & = \frac{y^{-2/\alpha}}{|\Gamma(-\tfrac{\alpha}{2})|} \int_0^\infty k_r(z) \, \frac{1}{(y^{-2/\alpha} r)^{1 + \alpha/2}} \exp\expr{-\frac{1}{4 (c_\alpha)^{2/\alpha} y^{-2/\alpha} r}} dr ,
}
which follows easily from the gamma integral, shows that $f \in \dom(L_H, x)$ and $L_H f(x) = L_B f(x)$ (this is a variant of the result proved in~\cite{bib:st10}).
\end{proof}

\begin{figure}
\begin{tikzcd}
\eqref{eq:D} \arrow[tail]{r} & \eqref{eq:I} \arrow[tail]{r} \arrow[tail]{rd} & \eqref{eq:S} \arrow[leftrightarrow]{d}{\alpha = 1} & \eqref{eq:B} \arrow{l} \arrow{ld} \\
\eqref{eq:II} \arrow[tail]{ur} & \eqref{eq:III} \arrow[tail]{u} & \eqref{eq:H} &
\end{tikzcd}
\caption{Relationship between pointwise definitions of $L f(x)$ for $f$ such that $(1 + |z|)^{-d - \alpha} f(x + z)$ is integrable. An arrow indicates inclusion of appropriate domains, and an arrow with a tail indicates proper inclusion.}
\end{figure}
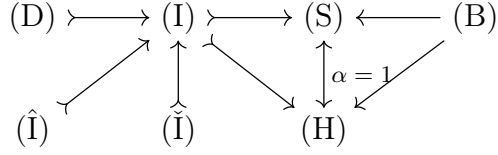

The proofs of pointwise results can be re-used for the corresponding statements for norm convergence in any of the spaces $\leb^p$, $p \in [1, \infty]$, $\conto$, $\contbu$ and $\contb$.

\begin{leftbar}
\begin{lemma}
\label{lem:norm}
Let $\X$ be any of the spaces $\leb^p$, $p \in [1, \infty]$, $\conto$, $\contbu$ and $\contb$. If $f \in \dom(L_D, \X)$, then $f \in \dom(L_I, \X)$, which in turn implies $f \in \dom(L_S, \X) \cap \dom(L_H, \X)$. Also, if $f \in \dom(L_B, \X)$, then $f \in \dom(L_S, \X) \cap \dom(L_H, \X)$. Furthermore, all definitions of $L f$ agree on appropriate domains.
\end{lemma}
\end{leftbar}

\begin{proof}
Below we only prove that $f \in \dom(L_I, \leb^p)$ implies $f \in \dom(L_S, \leb^p)$, the other statements being very similar. Recall that
\formula{
 p_t(z) & = t \int_0^\infty \nu_{t^{1/\alpha} r}(z) m'(r) dr ,
}
where $m'(r)$ is an absolutely integrable function, with integral $1$. Denote
\formula{
 \ph_r(x) & = \int_{\R^d} (f(x + z) - f(x)) \nu_r(z) dz .
}
If $f \in \dom(L_I, \leb^p)$, then $\ph_r$ converges to $\ph_{0+} = L_I f$ in $\leb^p$. By Fubini,
\formula{
 \int_{\R^d} (f(x + z) - f(x)) \, \frac{p_t(z)}{z} \, dz  - \ph_{0+}(x) & = \int_0^\infty (\ph_{t^{1/\alpha} r}(x) - \ph_{0+}(x)) m'(r) dr ,
}
and so, by dominated convergence,
\formula{
 \lim_{t \to 0^+} \|\tfrac{1}{t} (P_t f - f) - \ph_{0+}\|_p & \le \lim_{t \to 0^+} \int_0^\infty \|\ph_{t^{1/\alpha} r}(x) - \ph_{0+}(x)\|_p |m'(r)| dr = 0 ,
}
as desired.
\end{proof}

%
%

\section{M.~Riesz formulae}
\label{sec:riesz}

In the present section we provide further links between convergence in various Banach spaces, which are consequences of the following identity, discussed in detail in Section~\ref{sec:process}: for all $f \in \dom(L_S, \contbu)$ and all $y \in B_r$,
\formula[eq:dynkinb]{
 f(x + y) & = -\int_{B_r} L_S f(x + z) \gamma_r(y, z) dz + \int_{\R^d \setminus \overline{B}_r} f(x + z) \pi_r(y, z) dz ,
}
where
\formula[eq:pir]{
 \pi_r(y, z) & = \frac{2 \Gamma(\tfrac{d}{2})}{\alpha \, \pi^{d/2} \Gamma(\tfrac{\alpha}{2}) |\Gamma(-\tfrac{\alpha}{2})|} \, \frac{(r^2 - |y|^2)^{\alpha/2}}{|z|^d (|z|^2 - r^2)^{\alpha/2}}
}
is the \emph{Poisson kernel of a ball} $B_r$ for $L$, and
\formula[eq:gammar]{
 \gamma_r(y, z) & = \frac{\Gamma(\tfrac{d}{2})}{2^\alpha \pi^{d/2} (\Gamma(\tfrac{\alpha}{2}))^2} \, \frac{1}{|y - z|^{d - \alpha}} \, \int_0^{\tfrac{(r^2 - |y|^2) (r^2 - |z|^2)}{r^2 |y - z|^2}} \frac{s^{\alpha/2 - 1}}{(1 + s)^{d/2}} \, ds
}
is the \emph{Green function of a ball} $B_r$ for $L$. We record that $\pi_r(y, z) dz$ is a probability measure,
\formula{
 \pi_r(y, z) & = \int_{B_r} \gamma_r(y, v) \nu(z - v) dv ,
}
and
\formula{
 \int_{B_r} \gamma_r(y, z) dz & = \lim_{|z| \to \infty} \frac{\pi_r(y, z)}{\nu(z)} = \frac{\Gamma(\tfrac{d}{2})}{\alpha 2^{\alpha - 1} \Gamma(\tfrac{\alpha}{2}) \Gamma(\tfrac{d + \alpha}{2})} \, (r^2 - |y|^2)^{\alpha/2} .
}
When $\alpha < d$, the formulae for $\pi_r(y, z)$ and $\gamma_r(y, z)$ are all essentially due to M.~Riesz, see~\cite{bib:r38a,bib:r38b}. The case $\alpha \ge d$ is addressed in~\cite{bib:k57}, see also~\cite{bib:bgr61} and the references therein. For a detailed derivation of the expression for $\pi_r(y, z)$ and further properties, we refer to~\cite[Section~IV.5]{bib:l72} and~\cite[Section~V.4]{bib:bh86}, while the above form of $\gamma_r(y, z)$ was found in~\cite{bib:bgr61}. Formula~\eqref{eq:dynkinb} for $f \in \dom(L_S, \conto)$ is a very general result, valid for all generators of Feller semigroups, provided that $\pi_r(y, z)$ and $\gamma_r(y, z)$ are replaced by appropriate kernels, which are typically not given by closed-form expressions, see~\cite{bib:d65}.

We remark that by~\cite[formula~3.194.1]{bib:gr07}
\formula{
 \gamma_r(y, z) & = \frac{\Gamma(\tfrac{d}{2}) (r^2 - |y|^2)^{\alpha/2} (r^2 - |z|^2)^{\alpha/2}}{2^{\alpha - 1} \alpha \, \pi^{d/2} (\Gamma(\tfrac{\alpha}{2}))^2 r^\alpha |y - z|^d} \, {_2F_1}(\tfrac{d}{2}, \tfrac{\alpha}{2}; 1 + \tfrac{\alpha}{2}; -\tfrac{(r^2 - |y|^2) (r^2 - |z|^2)}{r^2 |y - z|^2}) .
}
For $y = 0$, by~\cite[formula~9.132.1]{bib:gr07}, when $\alpha \ne d$,
\formula{
 \gamma_r(0, z) & = \frac{c_{d,-\alpha}}{|z|^{d - \alpha}} - \frac{\Gamma(\tfrac{d}{2})}{2^{\alpha - 1} (d - \alpha) \pi^{d/2} (\Gamma(\tfrac{\alpha}{2}))^2} (r^2 - |z|^2)^{\alpha/2} {_2F_1}(\tfrac{d}{2}, 1; \tfrac{d - \alpha}{2}; \tfrac{|z|^2}{r^2}) .
}
and since ${_2F_1}(\tfrac{1}{2}, \tfrac{1}{2}; 1; z) = z^{-1/2} \arsinh z^{1/2}$, if $\alpha = d$ (and so $\alpha = d = 1$),
\formula{
 \gamma_r(0, z) & = \frac{1}{\pi} \, \arsinh \frac{\sqrt{r^2 - |z|^2}}{|z|} \, .
}
Similar expressions can be given for general $y \in B_r$.

In fact we do not need the explicit expressions for $\pi_r$ and $\gamma_r$; we are satisfied with the existence of $\pi_r(0, z)$ and $\gamma_r(0, z)$ such that~\eqref{eq:dynkinb} holds with $y = 0$, and the identity
\formula[eq:gammapinu]{
 \tilde{\nu}_r(z) & = \expr{\int_{B_r} \gamma_r(0, v) dv}^{-1} \pi_r(0, z) ,
}
with $\tilde{\nu}_r(z)$ defined by~\eqref{eq:tnu}. As remarked above, the appropriate kernels $\pi_r$ and $\gamma_r$ exist for every generator of a Feller semigroup, and~\eqref{eq:gammapinu} can be taken as the definition of $\tilde{\nu}_r(z)$. Therefore, the following result extends easily to much more general (at least translation-invariant) generators of Feller semigroups.

\begin{leftbar}
\begin{lemma}
\label{lem:dynkin}
Let $\X$ be any of the spaces $\leb^p$, $p \in [1, \infty)$, $\conto$ or $\contbu$. If $f \in \dom(L_S, \X)$, then $f \in \dom(L_D, \X)$, and $L_S f = L_D f$.
\end{lemma}
\end{leftbar}

\begin{proof}
When $\X = \conto$ or $\X = \contbu$, the result is a direct consequence of~\eqref{eq:dynkinb} and uniform continuity of $L_s f$: we have
\formula{
 & \int_{\R^d} (f(x + z) - f(x)) \tilde{\nu}_r(z) dz - L_S f(x) & \\
 & \hspace*{4em} = \expr{\int_{B_r} \gamma_r(0, z) dz}^{-1} \expr{\int_{\R^d \setminus \overline{B}_r} f(x + z) \pi_r(0, z) dz - f(x)} - L_S f(x) \\
 & \hspace*{4em} = \expr{\int_{B_r} \gamma_r(0, z) dz}^{-1} \int_{B_r} (L_S f(x + z) - L_S f(x)) \gamma_r(0, z) dz ,
}
and so
\formula[eq:dynkin:est]{
\begin{aligned}
 & \sup \set{\abs{\int_{\R^d} (f(x + z) - f(x)) \tilde{\nu}_r(z) dz - L_S f(x)} : x \in \R^d} \\
 & \hspace*{8em} \le \sup \{|L_S f(x + z) - L_S f(x)| : x \in \R^d, \, z \in B_r\} .
\end{aligned}
}
The right-hand side converges to $0$ as $r \to 0^+$, which proves the result for $\X = \conto$ or $\X = \contbu$.

Suppose that $\X = \leb^p$ for some $p \in [1, \infty)$. Let $g_\eps$ be a smooth approximate identity: $g_\eps(z) = \eps^{-d} g(\eps^{-1} z)$, where $g(z) \ge 0$, $\int_{\R^d} g(z) dz = 1$ and $g(z) = 0$ for $z \notin B$. Let $f \in \dom(L_S, \leb^p)$, and define $f_\eps = f * g_\eps$. By Fubini
\formula{
 \tfrac{1}{t} (P_t f_\eps - f_\eps) - L_S f * g_\eps & = (\tfrac{1}{t} (P_t f - f) - L_S f) * g_\eps .
}
Since the convolution with $g_\eps$, as an operator on $\leb^p$, has norm bounded by $\|g_\eps\|_1 = 1$, we have
\formula{
 \|\tfrac{1}{t} (P_t f_\eps - f_\eps) - L_S f * g_\eps\|_p & \le \|\tfrac{1}{t} (P_t f - f) - L_S f\|_p ,
}
and the right-hand side converges to $0$ as $t \to 0^+$. Therefore, $f_\eps \in \dom(L_S, \leb^p)$ and $L_S f_\eps = L_S f * g_\eps$.

In the above expressions $L_S$ is defined by~\eqref{eq:S}, with the limit in $\leb^p$. Observe, however, that $f_\eps \in \conto^\infty$, and hence, by Lemma~\ref{lem:smooth}, $f_\eps \in \dom(L_S, \conto)$. Since the limits in $\leb^p$ and $\conto$ coincide, we may write $L_S f_\eps = L_S f * g_\eps$, where $L_S f_\eps$ is defined by~\eqref{eq:S} with the limit in $\conto$, while $L_S f$ is defined by~\eqref{eq:S} with the limit in $\leb^p$. As in~\eqref{eq:dynkin:est}, using Minkowski's integral inequality,
\formula[eq:dynkin:aux]{
\begin{aligned}
 & \int_{\R^d} \abs{\int_{\R^d} (f_\eps(x + z) - f_\eps(x)) \tilde{\nu}_r(z) dz - L_S f_\eps(x)}^p dx \\
 & \hspace*{8em} \le \expr{\int_{B_r} \gamma_r(0, z) dz}^{-1} \int_{B_r} \int_{\R^d} |L_S f_\eps(x + z) - L_S f_\eps(x)|^p dx dz \\
 & \hspace*{8em} \le \sup \set{\int_{\R^d} |L_S f_\eps(x + z) - L_S f_\eps(x)|^p dx : z \in B_r} .
\end{aligned}
}
Recall that $f_\eps = f * g_\eps$ and $L_S f_\eps = L_S f * g_\eps$, so that as $\eps \to 0^+$, $f_\eps$ converges in $\leb^p$ to $f$, while $L_S f_\eps$ converges in $\leb^p$ to $L_S f$. Since $\tilde{\nu}_r \in \leb^1$, the convolution with $\tilde{\nu}_r$ is a continuous operator on $\leb^p$. Therefore, the left-hand side of~\eqref{eq:dynkin:aux} converges to
\formula{
 & \int_{\R^d} \abs{\int_{\R^d} (f(x + z) - f(x)) \tilde{\nu}_r(z) dz - L_S f(x)}^p dx .
}
On the other hand, the right-hand side of~\eqref{eq:dynkin:aux} is bounded above by
\formula{
 \sup \set{\int_{\R^d} |L_S f(x + z) - L_S f(x)|^p dx : z \in B_r} ,
}
which converges to $0$ as $r \to 0^+$, because $L_S f(x + z)$, as a function of $x$, converges in $\leb^p$ to $L_s f(x)$ when $z \to 0$. It follows that
\formula{
 \lim_{r \to 0^+} \int_{\R^d} \abs{\int_{\R^d} (f(x + z) - f(x)) \tilde{\nu}_r(z) dz - L_S f(x)}^p dx = 0 ,
}
as desired.
\end{proof}

As a direct consequence of Lemmas~\ref{lem:norm} and~\ref{lem:dynkin}, we obtain the following interesting statement, which appears to be partially new.

\begin{leftbar}
\begin{lemma}
\label{lem:ndis}
Let $\X$ be any of the spaces $\leb^p$, $p \in [1, \infty)$, $\conto$ or $\contbu$. Then the following conditions are equivalent: $f \in \dom(L_D, \X)$; $f \in \dom(L_I, \X)$; $f \in \dom(L_S, \X)$, and $L_D f = L_I f = L_S f$.
\end{lemma}
\end{leftbar}

When $\X = \conto$, the equivalence of~\eqref{eq:S} and~\eqref{eq:D} in the above proposition is a standard (and general) result, see~\cite[Theorem~5.5]{bib:d65}. Equivalence of~\eqref{eq:S} and~\eqref{eq:I} is also known, at least when $\X = \leb^p$ and $p \in [1, \tfrac{d}{\alpha})$, through the inversion of Riesz potential operators, see~\cite[Theorem~16.5 and Section 17]{bib:r96} and~\cite[Theorems~3.22 and~3.29]{bib:s01}.

%
%

\section{Norm convergence}
\label{sec:norm}

In this section we collect other results which connect various definitions of $L f$ for $f$ in $\leb^p$, $\conto$ and $\contbu$. Combined with Lemma~\ref{lem:ndis}, they prove the first part of Theorem~\ref{th:main}.

\begin{leftbar}
\begin{lemma}
\label{lem:l2sq}
The conditions $f \in \dom(L_S, \leb^2)$ and $f \in \dom(L_Q, \leb^2)$ are equivalent, and $L_S f = L_Q f$.
\end{lemma}
\end{leftbar}

\begin{proof}
By monotone convergence theorem, the quadratic forms corresponding to $\tfrac{1}{t} (I - P_t)$, namely
\formula{
 \form_t(f, f) & = \int_{\R^d} \expr{\int_{\R^d} (f(x) - f(x + z)) \, \frac{p_t(z)}{t} \, dz} \overline{f(x)} dx \\
 & = \frac{1}{2} \int_{\R^d} \int_{\R^d} |f(y) - f(x)|^2 \, \frac{p_t(y - x)}{t} \, dy dx ,
}
increase to the quadratic form $\form$, defined by~\eqref{eq:form}. If $\tfrac{1}{t} (P_t f - f)$ converges in $\leb^2$ to $L_S f$ as $t \to 0^+$, then $\form_t(f, f)$ converges to $-\int_{\R^d} L_S f(x) \overline{f(x)} dx$, and so $f \in \dom(\form)$. Furthermore, for any $g \in \dom(\form)$ (in fact, for any $g \in \leb^2$), $\form_t(f, g)$ converges to $-\int_{\R^d} L_S f(x) \overline{g(x)} dx$, and it follows that $f \in \dom(L_Q, \leb^2)$.

As in the proof of Lemma~\ref{lem:nsf}, the operator $L_Q$ defined by~\eqref{eq:Q}, with domain $\dom(L_Q, \leb^2)$, is an extension of $L_S$ defined by~\eqref{eq:S}, with domain $\dom(L_S, \leb^2)$. Since $-L_Q$ is non-negative definite, $\lambda I - L_Q$ is injective for any $\lambda > 0$. It follows that $\dom(L_Q, \leb^2)$ is equal to $\dom(L_S, \leb^2)$.
\end{proof}

Lemma~\ref{lem:l2sq} is a special case of a general result in the theory of Dirichlet forms, see~\cite[Sections~1.3 and~1.4]{bib:fot11}.

\begin{leftbar}
\begin{lemma}
\label{lem:nsf}
Let $p \in [1,2]$. Then $f \in \dom(L_S, \leb^p)$ if and only if $f \in \dom(L_F, \leb^p)$, and $L_S f = L_F f$.
\end{lemma}
\end{leftbar}

\begin{proof}
Let $\tfrac{1}{p} + \tfrac{1}{q} = 1$ and suppose that $f \in \dom(L_S, \leb^p)$. Then $\fourier (\tfrac{1}{t} (P_t f - f))(\xi) = \tfrac{1}{t} (e^{-t |\xi|^\alpha} - 1) \fourier f(\xi)$ converges as $t \to 0^+$ both in $\leb^q$ (because the Fourier transform is a bounded operator from $\leb^p$ to $\leb^q$) and pointwise. The two limit must coincide, that is, $\fourier (L_S f)(\xi) = |\xi|^\alpha \fourier f(\xi)$, as desired.

It follows that the fractional Laplace operator $L_F$ defined by~\eqref{eq:F}, with domain $\dom(L_F, \leb^p)$, is an extension of $L_S$ defined by~\eqref{eq:S}, with domain $\dom(L_S, \leb^p)$. Clearly, $\lambda I - L_F$ is injective (it is a Fourier multiplier with symbol $\lambda + |\xi|^\alpha$). It follows that $\dom(L_F, \leb^p) = \dom(L_S, \leb^p)$.
\end{proof}

Extension to $p \in (2, \infty)$ requires the distributional definition.

\begin{leftbar}
\begin{lemma}
\label{lem:nsw}
Let $\X$ be any of the spaces $\leb^p$, $p \in [1, \infty)$, $\conto$ or $\contbu$. Then $f \in \dom(L_S, \X)$ if and only if $f \in \dom(L_W, \X)$, and $L_S f = L_W f$.
\end{lemma}
\end{leftbar}

\begin{proof}
Suppose that $f \in \dom(L_S, \X)$ and let $\tilde{L}$ be the Schwartz distribution with Fourier transform $-|\xi|^\alpha$, as in~\eqref{eq:W}. We claim that if $\ph, \psi \in \schw$, then
\formula{
 (\tilde{L} * f) * (\ph * \psi) & = (\tilde{L} * \ph) * (f * \psi) \\
 & = \expr{\lim_{t \to 0^+} \tfrac{1}{t} (p_t * \ph - \ph)} * (f * \psi) \\
 & = \lim_{t \to 0^+} \expr{\tfrac{1}{t} (p_t * \ph - \ph) * (f * \psi)} \\
 & = \lim_{t \to 0^+} \expr{\tfrac{1}{t} (p_t * f - f) * (\ph * \psi)} \\
 & = \expr{\lim_{t \to 0^+} \tfrac{1}{t} (p_t * f - f)} * (\ph * \psi) .
}
Indeed, the first equality is the definition of the convolution of Schwartz distributions (and $\tilde{L}$ and $f$ are convolvable, because $\tilde{L} * \ph$ is integrable and $f * \psi$ is bounded). For the second one, observe that both $\tilde{L} * \ph$ and $\lim_{t \to 0^+} \tfrac{1}{t} (p_t * \ph - \ph)$ (the limit in $\leb^1$) have Fourier transforms $-|\xi|^\alpha \fourier \ph(\xi)$. To prove the third equality, note that $\tfrac{1}{t} (p_t * \ph - \ph)$ converges in $\leb^1$, and $f * \psi \in \leb^\infty$. The fourth equality follows by Fubini: $p_t, \ph, \psi \in \leb^1$ and $f \in \leb^1 + \leb^\infty$. Finally, the fifth one is a consequence of convergence of $\tfrac{1}{t} (p_t * f - f)$ in $\X$ and $\ph * \psi \in \leb^q$, where $\tfrac{1}{p} + \tfrac{1}{q} = 1$ (we take $p = \infty$ when $\X$ is $\conto$ or $\contbu$).

As in the proof of Lemma~\ref{lem:nsf}, the weak fractional Laplace operator $L_W$ defined by~\eqref{eq:W}, with domain $\dom(L_W, \X)$, is an extension of $L_S$ defined by~\eqref{eq:S}, with domain $\dom(L_S, \X)$. Since the Fourier transform of $L_W f$ is $-|\xi|^\alpha \fourier f(\xi)$ (note that the definition of multiplication here is not obvious, because $\fourier f$ is a distribution; we omit the details), $\lambda I - L_W$ is injective. Therefore, $\dom(L_W, \X) = \dom(L_S, \X)$.
\end{proof}

Lemma~\ref{lem:nsw} is rather well-known, as well as its extension to general translation-invariant generators of Markov semigroups (that is, generators of Lévy processes), see~\cite[Proposition~2.5]{bib:k11}. The first part of the argument, after obvious modification, gives the following result.

\begin{leftbar}
\begin{lemma}
\label{lem:nhw}
Let $\X$ be any of the spaces $\leb^p$, $p \in [1, \infty)$, $\conto$ or $\contbu$. Then $f \in \dom(L_H, \X)$ implies $f \in \dom(L_W, \X)$, and $L_H f = L_W f$.
\end{lemma}
\end{leftbar}

Finally, we recall, without proof, two results. The first one is a special case of a general theorem in the theory of fractional powers of dissipative operators.

\begin{leftbar}
\begin{theorem}[{\cite[Theorems~6.1.3 and~6.1.6]{bib:ms01}}]
\label{th:nsb}
Let $\X$ be any of the spaces $\leb^p$, $p \in [1, \infty)$, $\conto$ or $\contbu$. Then the following conditions are equivalent: $f \in \dom(L_S, \X)$; $f \in \dom(L_B, \X)$; $f \in \dom(L_\st{B}, \X)$. Furthermore, $L_S f = L_B f = L_\st{B} f$.
\end{theorem}
\end{leftbar}

The other one is the inversion formula for the Riesz potential operators.

\begin{leftbar}
\begin{theorem}[{\cite[Theorem~3.22]{bib:s01}}]
\label{th:nri}
Let $p \in [1, \tfrac{d}{\alpha})$. Then $f \in \dom(L_R, \leb^p)$ if and only if $f \in \dom(L_I, \leb^p)$, and $L_R f = L_I f$.
\end{theorem}
\end{leftbar}

Partial (distributional) extensions of the above result to $\leb^p$ for $p \in [\tfrac{d}{\alpha}, \infty)$ can be found in~\cite[Section~7.1]{bib:s01}. It is also of interest to study $f \in \leb^p$ for which $L f \in \leb^q$ with different $p$ and $q$, see~\cite[Sections~7.1, 7.3 and~7.4]{bib:s01}.

The results of this section, together with Lemma~\ref{lem:ndis}, prove the first part of Theorem~\ref{th:main}, which we state more formally below.

\begin{leftbar}
\begin{theorem}
\label{th:main1}
Let $\X$ be any of the spaces $\leb^p$, $p \in [1, \infty)$, $\conto$ or $\contbu$. Then the following conditions are equivalent:
\smallskip
{
\setlength{\multicolsep}{0em}
\setlength{\columnsep}{0.5em}
\begin{multicols}{2}
\begin{itemize}[leftmargin=1.5em]
\item $f \in \dom(L_F, \X)$ ($\X = \leb^p$, $p \in [1, 2]$);
\item $f \in \dom(L_W, \X)$;
\item $f \in \dom(L_B, \X)$;
\item $f \in \dom(L_\st{B}, \X)$;
\item $f \in \dom(L_I, \X)$;
\item $f \in \dom(L_D, \X)$;
\item $f \in \dom(L_Q, \X)$ ($\X = \leb^2$);
\item $f \in \dom(L_S, \X)$;
\item $f \in \dom(L_R, \X)$ ($\X = \leb^p$, $p \in [1, \tfrac{d}{\alpha})$);
\item $f \in \dom(L_H, \X)$.
\end{itemize}
\end{multicols}
}
\medskip\noindent
In addition, if $f \in \cont^1$ and $f \in \dom(L_\st{I}, \X)$, or if $f \in \dom(L_\stt{I}, \X)$, then $f$ satisfies all of the above conditions. Finally, all corresponding definitions of $L f$ agree.
\end{theorem}
\end{leftbar}

%
%

\section{Further results}
\label{sec:other}

In this section we collect results which relate pointwise and norm convergence in various definitions of $L$ on $\leb^p$, $\conto$ and $\contbu$.

It is a standard result that if $f \in \conto$, the limit in~\eqref{eq:S} exists for all $x$, and $L_S f(x)$ (defined pointwise) is a $\conto$ function, then in fact the limit in~\eqref{eq:S} is uniform; in other words, if $f \in \dom(L_S, x)$ for all $x$ and $L_S f(x)$ is in $\conto$, then $f \in \dom(L_S; \conto)$. Indeed, the operator $\tilde{L}$ defined by~\eqref{eq:S} for those $f \in \conto$, for which a pointwise limit in~\eqref{eq:S} exists for all $x$ and defines a $\conto$ function, is an extension of $L$ (with domain $\dom(L_S, \conto)$) which satisfies the positive maximum principle, and hence it is equal to $L$. Exactly the same argument, based on Proposition~\ref{prop:pmp}, proves the corresponding result for $\contbu$.

By Theorem~\ref{th:main1}, the above argument extends to pointwise convergence in other definitions as well, thus proving the next part of Theorem~\ref{th:main}, where everywhere pointwise convergence is discussed. For clarity, the result is formally stated below.

\begin{leftbar}
\begin{theorem}
\label{th:main2}
Let $\X$ be eather $\conto$ or $\contbu$. Then each of the equivalent conditions of Theorem~\ref{th:main1} is equivalent to each of the following conditions:
\begin{itemize}[leftmargin=1.5em]
\item $f \in \X$, $f \in \dom(L_B, x)$ for all $x$ and $L_B f \in \X$;
\item $f \in \X$, $f \in \dom(L_I, x)$ for all $x$ and $L_I f \in \X$;
\item $f \in \X$, $f \in \dom(L_D, x)$ for all $x$ and $L_D f \in \X$;
\item $f \in \X$, $f \in \dom(L_S, x)$ for all $x$ and $L_S f \in \X$;
\item $f \in \X$, $f \in \dom(L_H, x)$ for all $x$ and $L_H f \in \X$.
\end{itemize}
Furthermore, all corresponding definitions of $L f$ agree.
\end{theorem}
\end{leftbar}

The following example shows that continuity of $L f$ is essential. 

\begin{example}
Let $g(x) = e^{-|x|^2} x_1 / |x|$, where $x = (x_1, x_2, ..., x_d)$. Furthermore, let $\lambda > 0$ and $f = U_\lambda g = u_\lambda * g$, where $U_\lambda$ is the $\lambda$-resolvent operator and $u_\lambda$ is its kernel function (see Section~\ref{sec:S}). Then $f \in \conto$ and $f \in \dom(L_S, \leb^1)$. Furthermore, $f$ is smooth except at~$0$, and so $L_S f$ is defined both as a limit in $\leb^1$ and pointwise for $x \in \R^d \setminus \{0\}$. These two limits coincide almost everywhere, and, by continuity of $g$, everywhere in $\R^d \setminus \{0\}$. It follows that with $L_S f(x)$ defined pointwise,
\formula{
 L_S f(x) & = \lambda U_\lambda g(x) - (\lambda I - L_S) U_\lambda g(x) = \lambda f(x) - g(x)
}
for all $x \in \R^d \setminus \{0\}$. Since $f(-x) = -f(x)$, we also have $L_S f(0) = 0$. It follows that $f \in \conto$ and $f \in \dom(L_S, x)$ for all $x$, but $L_S f \notin \conto$, and therefore $f \notin \dom(L_S, \conto)$.
\end{example}

Boundedness of $f$ is essential when $\alpha \in (1, 2)$: clearly, any linear function $f$ satisfies then $L_S f(x) = 0$ for all $x$. When $\alpha \in (1, 2)$, also continuity of $f$ is essential, as indicated by the following surprising example. 

\begin{example}
Let $f(x) = |x_1|^{\alpha - 2} \sign x_1$, where $x = (x_1, x_2, ..., x_d)$. Then $f$ is locally integrable, and we claim that $L_D f(x) = 0$ for all $x$. Indeed, since $f(x + z) = -f(x - z)$ when $x_1 = 0$, we have $L_D f(x) = 0$ when $x_1 = 0$. Furthermore, $g(x) = |x_1|^{\alpha - 1}$ is known to satisfy $L_D g(x) = 0$ when $x_1 \ne 0$, and $f(x) = \tfrac{\partial}{\partial x_d} g(x)$. Our claim follows by an appropriate application of dominated convergence, we omit the details. In particular, $L_S f(x) = 0$ for all $x$. However, $f$ is not continuous, and clearly $f \notin \dom(L_S, \contbu)$.
\end{example}

Apparently, the above example in fact describes the worst case: if $f$ is more regular than above, we conjecture that in fact $f$ is continuous.

\begin{leftbar}
\begin{conjecture}
\label{con:sing}
Suppose that any of the following conditions is satisfied:
\begin{itemize}
\item $\alpha \in (0, 1]$ and $f$ is locally integrable;
\item $\alpha \in (1, 2)$, $q = \tfrac{1}{2 - \alpha}$ and $|f|^q$ is locally uniformly integrable;
\item $\alpha \in (0, 2]$, $f$ is locally integrable and $f \ge 0$.
\end{itemize}
If $f \in \dom(L_S, x)$ for all $x \in \R^d$ and $L_S f(x)$, defined pointwise by~\eqref{eq:S}, is a bounded uniformly continuous function, then $f$ is uniformly continuous the convergence in~\eqref{eq:S} is uniform. Furthermore, there is a linear function $g$ such that $f - g \in \dom(L_S, \contbu)$.
\end{conjecture}
\end{leftbar}

Almost everywhere convergence for functions in the $\leb^p$ domain follows by standard methods. The following result for $p \in [1, \tfrac{d}{\alpha})$ (and for $L_I$, not $L_S$; however, see Lemma~\ref{lem:aed} below for a version for $L_D$, $L_I$ and $L_H$) is proved in~\cite[Theorem~3.24]{bib:s01}.

\begin{leftbar}
\begin{lemma}
\label{lem:ae}
Let $p \in [1, \infty)$. If $f \in \dom(L_S, \leb^p)$, then $f \in \dom(L_S, x)$ for almost all $x$, and the pointwise limit in~\eqref{eq:S} is equal almost everywhere to the $\leb^p$ limit in~\eqref{eq:S}.
\end{lemma}
\end{leftbar}

\begin{proof}
Let $\lambda > 0$ and $g = \lambda f - L_S f$, so that $f = U_\lambda g$. Then
\formula{
 \tfrac{1}{t} (P_t f - f) - L_S f & = \tfrac{1}{t} (1 - e^{-\lambda t}) P_t f + \tfrac{1}{t} (e^{-\lambda t} P_t U_\lambda g - U_\lambda g) + (\lambda f - g) \\
 & = \tfrac{1}{t} (1 - e^{-\lambda t}) P_t f - \tfrac{1}{t} \int_0^t e^{-\lambda s} P_s g ds + (\lambda f - g) .
}
As a corollary of Lebesgue differentiation theorem, $P_t f$ converges almost everywhere to $f$ as $t \to 0^+$ and $P_s g$ converges almost everywhere to $g$ as $s \to 0^+$. Hence, the right-hand side of the above formula converges almost everywhere to $0$ as $t \to 0^+$; here we use the fact that the pointwise definition of the integral in the right-hand side coincides with the Bochner's definition of the integral in $\leb^p$.
\end{proof}

The converse is not true: for any $p \in [1, \infty)$ there is a function $f \in \leb^p$ such that $f \in \dom(L_S, x)$ for almost all $x$ and and $L_S f(x)$, defined pointwise by~\eqref{eq:S}, is a $\leb^p$ function, but $f$ is not in $\dom(L_S, \leb^p)$.

\begin{example}
There is a positive measure $\mu$, supported in a compact set $K$ of Lebesgue measure zero, such that $f = u_\lambda * \mu$ is in $\leb^1$ and $\leb^\infty$ (if $\alpha < d$, then, for example, $K$ can be an arbitrary set of positive $\alpha$-capacity and Lebesgue measure $0$, and $\mu$ its equilibrium measure, see~\cite[Section~II.1]{bib:l72}). Since $L_S u_\lambda(x) = \lambda u_\lambda(x)$ for $x \ne 0$, it follows that $L_S f(x) = \lambda f(x)$ for $x \in \R^d \setminus K$. However, $f$ is not in $\dom(L_S, \leb^p)$, for otherwise we would have $\lambda f - L_S f = 0$, and so $f = U_\lambda (\lambda f - L_S f) = 0$, a contradiction with $f > 0$ everywhere.
\end{example}

Convergence to an $\leb^p$ function everywhere is a different problem, see Conjecture~\ref{con:sing}.

Regularity results for functions in $\dom(L_S, \leb^p)$ (and also in spaces of Hölder continuous functions or in Besov spaces) follow easily from the identification of $\dom(L_S, \leb^p)$ with the space of Bessel potentials of $\leb^p$ functions (see~\cite[Theorem~7.16]{bib:s01}), and the properties of the latter (see~\cite[Chapter~V]{bib:s70}). We illustrate these results with the following simple statement, accompanied with a short proof. Note that by using the weak Young's inequality, one can slightly refine the last statement.

\begin{leftbar}
\begin{proposition}
\label{prop:cont}
Let $1 \le p < p' \le s' \le \infty$ and $0 < r < R < \infty$ or $r = R = \infty$. Suppose that $f \in \dom(L_S, \leb^p)$.
\begin{enumerate}[label=(\alph*)]
\item If $L_S f$ is continuous at some $x$, then $f$ is continuous at $x$.
\item If $\ind_{B(x, R)} L_S f \in \leb^{p'}$ and $p' > \tfrac{d}{\alpha}$, then $f$ is uniformly continuous in $B(x, r)$.
\item If $\ind_{B(x, R)} L_S f \in \leb^{p'}$ and $\tfrac{1}{s'} > \tfrac{1}{p'} - \tfrac{\alpha}{d}$, then $\ind_{B(x, r)} f \in \leb^{s'}$.
\end{enumerate}
Here we denote $B(x, \infty) = \R^d$.
\end{proposition}
\end{leftbar}

\begin{proof}
Let $\lambda > 0$ and $g \in \leb^p$, and suppose that $g_1 = \ind_{B(x, R)} g$ is in $\leb^{p'}$. Observe that $u_\lambda \in \leb^{q'}$, where $\tfrac{1}{p'} + \tfrac{1}{q'} = 1 + \tfrac{1}{s'}$; indeed, then $\tfrac{1}{q'} > \tfrac{d - \alpha}{d}$. By Young's inequality, $g_1 * u_\lambda \in \leb^{s'}$, and $g_1 * u_\lambda \in \contbu$ when $s' = \infty$.

Define $g_2 = \ind_{\R^d \setminus B(x, r)} g = g - g_1$. When $r = R = \infty$, then $g_2 = 0$. Suppose that $0 < r < R < \infty$. For $y \in B(x, r)$, we have
\formula{
 g_2 * u_\lambda(y) & = g_2 * (\ind_{\R^d \setminus B(0, R - r)} u_\lambda)(y) .
}
Furthermore, by the estimate~\eqref{eq:potest}, $\ind_{\R^d \setminus B(0, R - r)} u_\lambda \in \leb^q$, where $\tfrac{1}{p} + \tfrac{1}{q} = 1$, and hence $g_2 * (\ind_{\R^d \setminus B(0, R - r)} u_\lambda) \in \contbu$. We conclude that $\ind_{B(x, r)} (g * u_\lambda) \in \leb^{s'}$, and $g * u_\lambda$ is uniformly continuous in $B(x, r)$ when $s' = \infty$.

Suppose now that $g = L_S f$ for some $f \in \dom(L_S, \leb^p)$, and that $\ind_{B(x, R)} g \in \leb^{p'}$. Iterating the formula
\formula{
 f = (\lambda f - g) * u_\lambda = -g * u_\lambda + \lambda f * u_\lambda ,
}
we obtain
\formula{
 f = -\sum_{k = 1}^n \lambda^{k-1} g * u_\lambda^{*k} + \lambda^n f * u_\lambda^{*n}
}
for arbitrary $n \ge 1$, where $u_\lambda^{*k}$ is the convolution of $k$ factors $u_\lambda$. By the first part of the proof (and iteration), for any $k \ge 1$ we have $\ind_{B(x, r)} (g * u_\lambda^{*k}) \in \leb^{s'}$, and $g * u_\lambda^{*k}$ is uniformly continuous in $B(x, r)$ when $s' = \infty$. Clearly, $u_\lambda^{*n} \in \leb^1$. Furthermore, if $n > \tfrac{d}{\alpha}$, then $\fourier u_\lambda^{*n}(\xi) = (\lambda + |\xi|^\alpha)^{-n}$ is integrable, and so $u_\lambda^{*n} \in \leb^\infty$. It follows that $u_\lambda^{*n} \in \leb^q$, where $\tfrac{1}{p} + \tfrac{1}{q} = 1$, and hence $f * u_\lambda^{*n} \in \contbu$. This proves~(b) and~(c). Clearly, (a) is a special case of~(b).
\end{proof}

%
%

\section{Isotropic stable Lévy process}
\label{sec:process}

A detailed treatment of isotropic stable Lévy processes can be found, for example, in~\cite{bib:bh86,bib:bbkrsv09}. Here we only give an informal introduction to the subject and discuss some benefits of the probabilistic approach.

For each starting point $x$ there exists a stochastic process $X_t$, $t \ge 0$, with the corresponding probability measure and expectation denoted by $\pr_x$ and $\ex_x$, with the following properties:
\begin{enumerate}[label=(\alph*)]
\item $X_t$ starts at $x$, that is, $\pr_x(X_0 = x) = 1$;
\item $X_t$ has independent increments, that is, $X_{t_n} - X_{s_n}$ are independent whenever $0 \le s_1 < t_1 \le s_2 < t_2 \le ...$;
\item $X_t$ has right-continuous paths with left limits, that is, the one-sided limits $X_{t+}$ and $X_{t-}$ exist and $X_{t+} = X_t$ for all $t$;
\item the distribution of $X_t$ under $\pr_x$ is equal to $p_t(x + z) dz$.
\end{enumerate}
We remark that a process satisfying conditions~(a) through~(c) is said to be a \emph{Lévy process}. A process $X_t$ is said to be a \emph{Feller process} if its state space is a locally compact metrisable space, it satisfies~(a) and (c), and also the following two conditions (which are weaker than~(b)): $X_t$ has the Markov property, and the transition operators $T_t f(x) = \ex_x f(X_t)$ form a strongly continuous semigroup on $\conto$. A process $X_t$ is \emph{isotropic} if it is invariant under orthogonal transformation of the state space. Finally, $X_t$ is \emph{stable} if the process $c X_t$ under $\pr_x$ has the same law as the process $X_{c^\alpha t}$ under $\pr_{c x}$.

A random time $\tau$ is said to be a \emph{Markov time} for $X_t$ if the event $\{\tau < t\}$ is measurable with respect to the (appropriately augmented) $\sigma$-algebra of sets generated by the family of random variables $\{X_s : s \in [0, t]\}$; intuitively this means that in order to determine whether $\tau < t$ it suffices to see the path of the process up to time $t$. It is known that first exit times of open sets:
\formula{
 \tau_D = \inf \{ t \ge 0 : X_t \notin D \}
}
are Markov times (in fact this is true whenever $D$ is a Borel set).

By~(d), the transition operators of $X_t$ are the operators $P_t$ introduced in Section~\ref{sec:S}, that is, $\ex_x f(X_t) = P_t f(x)$. Therefore, by Fubini, the $\lambda$-potential operators of $X_t$ are the $\lambda$-resolvent operators $U_\lambda$:
\formula{
 U_\lambda f(x) & = \int_0^\infty e^{-\lambda t} P_t f(x) dt = \ex_x \int_0^\infty e^{-\lambda t} f(X_t) dt .
}
Below we state Dynkin's formula, one of the fundamental results in the theory of Markov processes, in a version valid for an arbitrary Feller process. For non-random time $\tau$ it reduces to the well-known formula
\formula{
 P_t f & = f + \int_0^t P_s L_S f ds .
}
Its full strength is, however, presented when $\tau$ is the first exit time of a set.

\begin{leftbar}
\begin{theorem}[{Dynkin's formula, \cite[Theorem~5.1]{bib:d65}}]
\label{th:dynkin}
If $\tau$ is a Markov time, $\lambda \ge 0$, $g$ is measurable and
\formula{
 f(x) = \ex_x \int_0^\infty e^{-\lambda t} g(X_t) dt ,
}
then for all $x$ for which the integral in the right-hand side of the above formula is absolutely convergent,
\formula[eq:dynkin0]{
 \ex_x (e^{-\lambda \tau} f(X_\tau)) & = f(x) - \ex_x \int_0^\tau e^{-\lambda s} g(X_s) ds .
}
In particular, if $\lambda > 0$, $f \in \dom(L_S, \contbu)$ and $\lambda > 0$, then for all $x$,
\formula[eq:dynkin1]{
 \ex_x (e^{-\lambda \tau} f(X_\tau)) & = f(x) - \ex_x \int_0^\tau e^{-\lambda s} (\lambda I - L_S) f(X_s) ds .
}
If in addition $\ex_x \tau < \infty$ (or, more generally, if $\ex_x(\int_0^\tau |L_S f(X_s)| ds) < \infty$), then the same formula holds with $\lambda = 0$, that is,
\formula[eq:dynkin]{
 \ex_x f(X_\tau) & = f(x) + \ex_x \int_0^\tau L_S f(X_s) ds .
}
\end{theorem}
\end{leftbar}

\begin{proof}[Sketch of the proof]
Formula~\eqref{eq:dynkin0} follows by splitting the integral defining $U_\lambda g(x)$ into $\int_0^\tau$ and $\int_\tau^\infty$, and applying the strong Markov property for the latter one
; we omit the details. If $f \in \dom(L_S, \contbu)$, then $f = U_\lambda g$ for $g = \lambda f - L f$, and so formula~\eqref{eq:dynkin1} is merely a reformulation of~\eqref{eq:dynkin0}. Formula~\eqref{eq:dynkin} follows by dominated convergence.
\end{proof}

The identity~\eqref{eq:dynkinb} is simply the Dynkin's formula~\eqref{eq:dynkin} applied to the first exit time from a ball $B(x, r)$, together with explicit expressions for the expectation of $f(X_\tau)$ and $\int_0^\tau g(X_s) ds$:
\formula{
 \ex_x f(X_{\tau_{B(x, r)}}) & = \int_{\R^d \setminus \overline{B}_r} f(x + z) \pi_r(0, z) dz , \\
 \ex_x \int_0^{\tau_{B(x, r)}} g(X_s) ds & = \int_{B_r} g(x + z) \gamma_r(0, z) dz ,
}
where the Poisson kernel $\pi_r$ of a ball is given by~\eqref{eq:pir}, and the Green function of a ball is given by~\eqref{eq:gammar}.

We remark that for general Feller processes, $\ex_x \tau_D < \infty$ provided that $D$ is a sufficiently small neighbourhood of $x$ and $X_t$ is not constantly equal to $x$ under $\pr_x$ (that is, $x$ is not an absorbing state). In our case in fact $\ex_x \tau_D < \infty$ for all bounded $D$.

The \emph{Dynkin characteristic operator}, defined by the formula
\formula[eq:dynkinop]{
 L_D f(x) & = \lim_{r \to 0^+} \frac{\ex_x f(X_{\tau_{B(x, r)}}) - f(x)}{\ex_x \tau_{B(x, r)}}
}
for all functions $f$ for which the limit exists, is a way to localize the definition of $L_S$ for a general Feller process. In our case, \eqref{eq:dynkinop} reduces to~\eqref{eq:D}. The Dynkin characteristic operator is particularly useful when the state space is not the full space $\R^d$ and distribution theory cannot be used: by a general result, $L_D$ defined by~\eqref{eq:dynkinop} is an extension of $L_S$ with domain $\dom(L_S, \conto)$, and if $f \in \conto$ and $L_D f \in \conto$ (with $L_D f$ defined pointwise by~\eqref{eq:dynkinop}), then in fact $f \in \dom(L_S, \conto)$. We have seen, however, that even for the fractional Laplace operator in the full space $\R^d$ the Dynkin characteristic operator is a convenient way to prove the equivalence of the singular integral and semigroup definitions of $L$.

Using probabilistic methods, one immediately obtains the following simple and well-known, but quite useful result, which is a pointwise version of formula~\eqref{eq:li}.

\begin{proposition}
\label{prop:pot:gen}
If $\alpha < d$, $|x - y|^{\alpha - d} g(y)$ is absolutely integrable in $y \in \R^d$, and $g$ is continuous at $x$ (or, more generally, $x$ is a Lebesgue point of $g$), then $I_\alpha g$ is in $\dom(L_D, x)$, and $L_D I_\alpha g(x) = -g(x)$ (where $I_\alpha$ is the Riesz potential operator, that is, the $0$-resolvent operator $U_0$ of the semigroup $P_t$). Consequently, a similar statement holds true for $L_I$, $L_S$ and $L_H$.
\end{proposition}

\begin{proof}
By Lemmas~\ref{lem:dip}, \ref{lem:pis} and~\ref{lem:pih}, it suffices to prove the result for $L_D$. By the assumption, the integral $f(y) = I_\alpha g(y) = \ex_y \int_0^\infty g(X_s) ds$ is absolutely convergent when $y = x$, and so we may use Dynkin's formula~\eqref{eq:dynkin0} with $\lambda = 0$. It follows that
\formula{
 L_D f(x) & = \lim_{r \to 0^+} \frac{\ex_x f(X_{\tau_{B(x, r)}}) - f(x)}{\ex_x \tau_{B(x, r)}} \\
 & = -\lim_{r \to 0^+} \frac{1}{\ex_x \tau_{B(x, r)}} \, \ex_x \int_0^{\tau_{B(x, r)}} g(X_s) ds \\
 & = -\lim_{r \to 0^+} \frac{1}{\ex_0 \tau_B} \, \ex_0 \int_0^{\tau_B} g(x + r X_t) dt \\
 & = -\lim_{r \to 0^+} \expr{\int_B \gamma_1(0, y) dy}^{-1} \int_B g(x + r y) \gamma_1(0, y) dy .
}
If $g$ is continuous at $x$, or if $x$ is a Lebesgue point of $g$, then the limit in the right-hand side is equal to $g(x)$, as desired.
\end{proof}

We conclude this article with another application of the probabilistic method, which proves the last statement of Theorem~\ref{th:main}. It is a version of Lemma~\ref{lem:ae} for the Dynkin's definition $L_D$ of the fractional Laplace operator.

\begin{leftbar}
\begin{lemma}
\label{lem:aed}
Let $p \in [1, \infty)$. If $f \in \dom(L_S, \leb^p)$, then $f \in \dom(L_D, x)$, $f \in \dom(L_I, x)$, $f \in \dom(L_S, x)$ and $f \in \dom(L_H, x)$ for almost all $x$, and for all $x$ at which $L_S f$ is continuous. Furthermore, the pointwise limits in~\eqref{eq:D}, \eqref{eq:I}, \eqref{eq:S} and~\eqref{eq:H} are equal almost everywhere to the $\leb^p$ limit in~\eqref{eq:S}.
\end{lemma}
\end{leftbar}

\begin{proof}
As before, by Lemmas~\ref{lem:dip}, \ref{lem:pis} and~\ref{lem:pih}, it suffices to prove the result for $L_D$. Let $\lambda > 0$, $f \in \dom(L_S, \leb^p)$ and $g = L_S f$, so that $f = U_\lambda (\lambda f - g)$. Observe for almost all $x$, the integrals defining $U_\lambda f(x) = \ex_x \int_0^\infty e^{-\lambda t} f(X_t) dt$ and $U_\lambda g(x) = \ex_x \int_0^\infty e^{-\lambda t} g(X_t) dt$ are absolutely convergent for all $\lambda > 0$. For such a point $x$ and for $r > 0$, by Dynkin's formula~\eqref{eq:dynkin0},
\formula{
 \frac{\ex_x f(X_{\tau_{B(x, r)}}) - f(x)}{\ex_x \tau_{B(x, r)}} & = \frac{\ex_x ((1 - e^{-\lambda \tau_{B(x, r)}}) f(X_{\tau_{B(x, r)}}))}{\ex_x \tau_{B(x, r)}} \\
 & \hspace*{5em} - \frac{1}{\ex_x \tau_{B(x, r)}} \, \ex_x \int_0^{\tau_{B(x, r)}} e^{-\lambda s} (\lambda f - g)(X_s) ds .
}
In particular, the integrals in the right-hand side are absolutely convergent. Consider the limit $\lambda \to 0^+$. In the former integral in the right-hand side, $1 - e^{\lambda \tau_{B(x, r)}}$ decreases to $0$, and so the integral converges to $0$. The latter integral is a convolution of $\lambda f - g$ with a positive kernel function, which increases as $\lambda \to 0^+$ to the integrable kernel function $\gamma_r(0, z)$. It follows that for almost all $x$ one can pass to the limit under the integral sign, and hence, as in the proof of Proposition~\ref{prop:pot:gen},
\formula{
 \frac{\ex_x f(X_{\tau_{B(x, r)}}) - f(x)}{\ex_x \tau_{B(x, r)}} & = \frac{1}{\ex_x \tau_{B(x, r)}} \, \ex_x \int_0^{\tau_{B(x, r)}} g(X_s) ds \\
 & = \expr{\int_B \gamma_1(0, y) dy}^{-1} \int_B g(x + r y) \gamma_1(0, y) dy .
}
By Lebesgue's differentiation theorem, for almost all $x$ the right-hand side converges to $g(x)$ as $r \to 0^+$, as desired.

Finally, note that every time in the above argument a condition is satisfied for almost all $x$, it is also satisfied for all $x$ at which $f$ and $g$ are continuous. Furthermore, by Proposition~\ref{prop:cont}, continuity of $g$ at some $x$ implies continuity of $f$ at $x$.
\end{proof}

%
%

\section*{Acknowledgements}

I thank Krzysztof Bogdan, Bartłomiej Dyda and René Schilling for stimulating discussions on the subject of this article.

%
%

%
%


\begin{thebibliography}{00}

\bibitem{bib:bk04}
R.~Bañuelos, T.~Kulczycki,
\emph{The Cauchy process and the Steklov problem}.
J.~Funct. Anal. 211(2) (2004): 355--423.

\bibitem{bib:bk06}
R.~Bañuelos, T.~Kulczycki,
\emph{Spectral gap for the Cauchy process on convex symmetric domains}.
Comm. Partial Diff. Equations 31 (2006): 1841--1878.

\bibitem{bib:b98}
J.~Bertoin,
\emph{Lévy Processes}.
Cambridge University Press, Melbourne-New York, 1998.

\bibitem{bib:bh86}
J.~Bliedtner, W.~Hansen,
\emph{Potential theory, an analytic and probabilistic approach to balayage}.
Springer-Verlag, 1986.

\bibitem{bib:bgr61}
R.~M.~Blumenthal, R.~K.~Getoor, D.~B.~Ray,
\emph{On the distribution of first hits for the symmetric stable processes}.
Trans. Amer. Math. Soc. 99 (1961): 540--554.

\bibitem{bib:bbc03}
K.~Bogdan, K.~Burdzy, Z.-Q. Chen,
\emph{Censored stable processes}.
Probab. Theory Related Fields 127(1) (2003): 89--152.

\bibitem{bib:bb99}
K.~Bogdan, T.~Byczkowski,
\emph{Potential theory for the $\alpha$-stable Schrödinger operator on bounded Lipschitz domains}.
Studia Math. 133(1) (1999): 53--92.

\bibitem{bib:bbkrsv09}
K.~Bogdan, T.~Byczkowski, T.~Kulczycki, M.~Ryznar, R.~Song, Z.~Vondraček,
\emph{Potential Analysis of Stable Processes and its Extensions}.
Lecture Notes in Mathematics 1980, Springer, 2009.

\bibitem{bib:bkk08}
K.~Bogdan, T.~Kulczycki, M.~Kwaśnicki,
\emph{Estimates and structure of $\alpha$-harmonic functions}.
Probab. Theory Related Fields 140(3--4) (2008): 345--381.

\bibitem{bib:bkk15}
K.~Bogdan, T.~Kumagai, M.~Kwaśnicki,
\emph{Boundary Harnack inequality for Markov processes with jumps}.
Trans. Amer. Math. Soc. 367(1) (2015): 477--517.

\bibitem{bib:bz06}
K.~Bogdan, T.~Żak,
\emph{On Kelvin transformation}.
J.~Theor. Prob. 19(1) (2006): 89--120.

\bibitem{bib:bv15}
C.~Bucur, E.~Valdinoci,
\emph{Non-local diffusion and applications}.
Preprint, 2015, arXiv:1504.08292.

\bibitem{bib:cs07}
L.~Caffarelli, L.~Silvestre,
\emph{An extension problem related to the fractional Laplacian}.
Comm. Partial Differential Equations 32(7) (2007): 1245--1260.

\bibitem{bib:css08}
L.~Caffarelli, S.~Salsa, L.~Silvestre,
\emph{Regularity estimates for the solution and the free boundary of the obstacle problem for the fractional Laplacian}.
Inventiones Math. 171(2) (2008): 425--461.


\bibitem{bib:d90}
R.~D.~DeBlassie,
\emph{The first exit time of a two-dimensional symmetric stable process from a wedge}.
Ann. Probab. 18 (1990), pp.~1034--1070.

\bibitem{bib:d04}
R.~D.~DeBlassie,
\emph{Higher order PDE's and symmetric stable processes}.
Probab. Theory Related Fields 129 (2004), pp.~495--536.

\bibitem{bib:dm07}
R.~D.~DeBlassie, P.~J.~Méndez-Hernández,
\emph{$\alpha$-continuity properties of the symmetric $\alpha$-stable process}.
Trans. Amer. Math. Soc. 359 (2007), pp.~2343--2359.

\bibitem{bib:ds53}
M.~Dunford, J.~T.~Schwartz,
\emph{Linear Operators. General theory}.
Interscience Publ., 1953.

\bibitem{bib:d12}
B.~Dyda,
\emph{Fractional calculus for power functions and eigenvalues of the fractional Laplacian}.
Fract. Calc. Appl. Anal. 15(4) (2012): 536--555.

\bibitem{bib:dkk15}
B.~Dyda, A.~Kuznetsov, M.~Kwaśnicki,
\emph{Fractional Laplace operator and Meijer G-function}.
In preparation.

\bibitem{bib:d65}
E.~B.~Dynkin,
\emph{Markov processes, Vols. I and II}.
Springer-Verlag, Berlin-G{\"o}tingen-Heidelberg, 1965.

\bibitem{bib:fk83}
D.~W.~Fox, J.~R.~Kuttler,
\emph{Sloshing frequencies}.
Z. Angew. Math. Phys. 34 (1983): 668--696.

\bibitem{bib:fl48}
K.~O.~Friedrichs, H.~Lewy.
\emph{The dock problem}.
Commun. Pure Appl. Math. 1 (1948): 135--148.

\bibitem{bib:fot11}
M.~Fukushima, Y.~Oshima, M.~Takeda,
\emph{Dirichlet Forms and Symmetric Markov Processes}.
De Gruyter, 2011.

\bibitem{bib:gms13}
J.~E.~Gal\'e, P.~J.~Miana, P.~R.~Stinga,
Extension problem and fractional operators: semigroups and wave equations.
J.~Evol. Equations 13 (2013): 343--368.

\bibitem{bib:g61}
R.~K.~Getoor,
\emph{First Passage Times for Symmetric Stable Processes in Space}.
Trans. Amer. Math. Soc. 101(1) (1961): 75--90.

\bibitem{bib:gr07}
I.~S.~Gradshteyn, I.~M.~Ryzhik,
\emph{Table of Integrals, Series and Products}.
Academic Press, 2007.

\bibitem{bib:h64}
R.~L.~Holford,
\emph{Short surface waves in the presence of a finite dock. I, II}.
Proc. Cambridge Philos. Soc. 60 (1964): 957--983, 985--1011.

\bibitem{bib:j01}
N.~Jacob,
\emph{Pseudo Differential Operators and Markov Processes. Vol. 1}.
Imperial College Press, London, 2001

\bibitem{bib:k57}
M.~Kac,
\emph{Some remarks on stable processes}.
Publ. Inst. Statist. Univ. Paris 6 (1957): 303--306.

\bibitem{bib:kk04}
V.~Kozlov, N.~G.~Kuznetsov,
\emph{The ice-fishing problem: the fundamental sloshing frequency versus geometry of holes}.
Math. Meth. Appl. Sci. 27 (2004): 289--312.

\bibitem{bib:kkms10}
T.~Kulczycki, M.~Kwaśnicki, J.~Małecki, A.~Stós,
\emph{Spectral Properties of the Cauchy Process on Half-line and Interval}.
Proc. London Math. Soc. 30(2) (2010): 353--368.

\bibitem{bib:k11}
M.~Kwaśnicki,
\emph{Spectral analysis of subordinate Brownian motions on the half-line}.
Studia Math. 206(3) (2011): 211--271.

\bibitem{bib:l72}
N.~S.~Landkof,
\emph{Foundations of Modern Potential Theory}.
Springer, New York--Heidelberg, 1972.

\bibitem{bib:l83}
W.~Luther,
\emph{Abelian and Tauberian theorems for a class of integral transforms}.
J.~Math. Anal. Appl. 96(2) (1983): 365--387.

\bibitem{bib:ms01}
C.~Martínez, M.~Sanz,
\emph{The Theory of Fractional Powers of Operators}.
North-Holland Math. Studies 187, Amsterdam, 2001.

\bibitem{bib:mo69}
S.~A.~Molchanov, E.~Ostrowski,
\emph{Symmetric stable processes as traces of degenerate diffusion processes}.
Theor. Prob. Appl. 14(1) (1969): 128--131.

\bibitem{bib:r38a}
M.~Riesz,
\emph{Intégrales de Riemann--Liouville et potentiels}.
Acta Sci. Math. Szeged 9 (1938): 1--42.

\bibitem{bib:r38b}
M.~Riesz,
\emph{Rectification au travail ``Intégrales de Riemann--Liouville et potentiels''}.
Acta Sci. Math. Szeged 9 (1938): 116--118.

\bibitem{bib:ro15}
X.~Ros-Oton,
\emph{Nonlocal elliptic equations in bounded domains: a survey}.
Preprint, 2015, arXiv:1504.04099.

\bibitem{bib:r96}
B.~Rubin,
\emph{Fractional Integrals and Potentials}.
Monographs and Surveys in Pure and Applied Mathematics 82, Chapman and Hall/CRC, 1996.

\bibitem{bib:s01}
S.~Samko,
\emph{Hypersingular Integrals and Their Applications}.
CRC Press, 2001.

\bibitem{bib:s99}
K.~Sato,
\emph{Lévy Processes and Infinitely Divisible Distributions}.
Cambridge Univ. Press, Cambridge, 1999.

\bibitem{bib:ssv12}
R.~Schilling, R.~Song, Z.~Vondraček,
\emph{Bernstein Functions: Theory and Applications}.
De Gruyter, Studies in Math. 37, Berlin, 2012.

\bibitem{bib:ss75}
K.~Soni, R.~P.~Soni,
\emph{Slowly Varying Functions and Asymptotic Behavior of a Class of Integral Transforms I, II, III}.
J. Anal. Appl. 49 (1975): 166--179; 477--495; 612--628.

\bibitem{bib:s58}
F.~Spitzer,
\emph{Some theorems concerning 2-dimensional Brownian motion}.
Trans. Amer. Math. Soc. 87 (1958): 187--197.

\bibitem{bib:s70}
E.~M.~Stein,
\emph{Singular Integrals And Differentiability Properties Of Functions}.
Princeton University Press, 1970.

\bibitem{bib:st10}
P.~R.~Stinga, J.~L.~Torrea,
Extension Problem and Harnack's Inequality for Some Fractional Operators.
Comm. Partial Diff. Equations 35 (2010): 2092--2122.

\end{thebibliography}
\end{document}